\documentclass[reqno]{amsart}
\usepackage{hyperref}

\begin{document}
\title[\hfilneg
\hfil Multiplicity of positive solutions for a fractional Laplacian ]
{Multiplicity of positive solutions for a fractional Laplacian equations involving critical nonlinearity}

\author[J. Zhang\,\,,X. Liu\,\,, H. Jiao\hfil
\hfilneg]
{Jinguo Zhang,\quad Xiaochun Liu,\quad Hongying Jiao}

\address{Jinguo Zhang\newline
   School of Mathematics \\
  Jiangxi Normal University \\
  330022 Nanchang, China}
\email{yuanxin1027suda@163.com}

\address{Xiaochun Liu \newline
   School of Mathematics and Statistics \\
  Wuhan University \\
  430072 Wuhan, China}
\email{xcliu@whu.edu.cn}

\address{Hongying Jiao \newline
School of Sciences\\
Air Force Engineering University\\
Xi'an 710051, China}
\email{HYJiao12@163.com}

\thanks{Supported by NSFC Grant No.11371282.}
\subjclass[2010]{35J60, 47J30}
 \keywords{Fractional Laplacian equation; Critical Sobolev exponent; Variational methods.}

\begin{abstract}
In this paper we deal with the multiplicity of positive solutions to the fractional Laplacian equation
 \begin{equation*}
 \left\{\aligned
(-\Delta)^{\frac{\alpha}{2}} u&=\lambda f(x)|u|^{q-2}u+|u|^{2^{*}_{\alpha}-2}u, \quad &\text{in}\,\,\Omega,\\
u&=0,&\text{on}\,\,\partial\Omega,\\
 \endaligned\right.
 \end{equation*}
where $\Omega\subset \mathbb{R}^{N}(N\geq 2)$ is a bounded domain with smooth boundary, $0<\alpha<2$,
$(-\Delta)^{\frac{\alpha}{2}}$ stands for the fractional Laplacian operator,
$f\in C(\Omega\times\mathbb{R},\mathbb{R})$ may be sign changing  and $\lambda$ is a positive parameter.
We will prove that there exists $\lambda_{*}>0$ such that
the problem has at least two positive solutions for each $\lambda\in (0\,,\,\lambda_{*})$.
In addition, the concentration behavior of the solutions are investigated.
 \end{abstract}

\maketitle
\numberwithin{equation}{section}
\newtheorem{theorem}{Theorem}[section]
\newtheorem{lemma}{Lemma}[section]
\newtheorem{remark}{Remark}[section]
\newtheorem{proposition}{Proposition}[section]
\newtheorem{corollary}{Corollary}[section]
\newtheorem{definition}{Definition}[section]
\newcommand{\R}{\mathbb{R}^{N}}

\newcommand{\HO}{H^{s}_{0}(\Omega)}
\newcommand{\HHO}{H^{s}_{0,L}(\mathcal{C})}
\newcommand{\HHR}{H^{s}_{0,L}(\Omega)}
\newcommand{\HR}{H^{s}_{0}(\Omega)}
\section{Introduction}
In this paper, we are concerned with the multiplicity of solutions to the
following fractional Laplacian equation:
\begin{equation}\label{eq01}
 \left\{\aligned
(-\Delta)^{\frac{\alpha}{2}} u&=\lambda f(x)|u|^{q-2}u+|u|^{2^{*}_{\alpha}-2}u, \quad &\text{in}\,\,\Omega,\\
u&=0,&\text{on}\,\,\partial\Omega,\\
 \endaligned\right.
 \end{equation}
where $\Omega \subset \mathbb{R}^{N}$, $N\geq 2$, is a bounded domain with
smooth boundary, $0<\alpha<2$,  $(-\Delta)^{\frac{\alpha}{2}}$ stands for
the fractional Laplacian operator, $2^{*}_{\alpha}:=\frac{2N}{N-\alpha}$,
$1<q<2$, $\lambda>0$ and $f:\Omega\to \mathbb{R}$ is a continuous function with
 $f^{+}(x)=\max\{f(x),0\}\neq0$ on $\Omega$. From the assumptions on $f$
 and $q$, we know that the problem \eqref{eq01} involving
 the concave-convex nonlinearities and sign-changing weight function.

The fractional power of the Laplacian is the infinitesimal generators
of L\'{e}vy stable diffusion process and arises in
anomalous diffusions in plasmas, flames propagation and chemical reactions
in liquids, population dynamics, geophysical
fluid dynamics and American options in finance. For more details,
 one can see \cite{da2004,as2005}
and references therein.

Recently the fractional Laplacian attracts much interest in nonlinear analysis,
 such as in \cite{bc2012,bc2013,ll2007,xj2005,xt2010,se2013-1,se2013-2,szy2014,t2011}.
Caffarelli and Silvestre \cite{ll2007} gave a new formulation of the fractional
 Laplacians through Dirichlet-Neumann maps.
This is commonly used in the recent literature since it allows us to write nonlocal
 problems in a local way and this permits to us use the
variational methods for those kinds of problems.
In \cite{xt2010}, Cabr\'{e} and Tan defined the operator of the square root of
Laplacian through the spectral decomposition
of the Laplacian operator on $\Omega$ with zero Dirichlet boundary
conditions. With classical local techniques, they established existence
 of positive solutions for problems with subcritical
 nonlinearities, regularity and $L^{\infty}$-estimate of Brezis-Kato
 type for weak solutions. In \cite{se2013-1,se2013-2,t2011},
 the authors employed the Brezis-Nirenberg technique to build an analogue
 results to the problem in \cite{bn1983}, but with the fractional Laplacian
 instead of the Laplacian.

The analogue problem to problem \eqref{eq01} for the Laplacian operator has
 been investigated widely in the past decades,
 see for example \cite{dh1996,ct1992,wu2008} and the references therein.
 The main purpose of this paper is to generalize the partial results of
 \cite{bc2012} to the problem involving sign-changing weight function.
 Using the variational methods and the Nehari
 manifold decomposition, we first prove that the problem \eqref{eq01} has
 at least two positive solutions for $\lambda$ sufficiently small.

\begin{theorem}\label{th01}
There exists $\lambda_{*}>0$ such that for $\lambda\in (0\,,\,\lambda_{*})$,
 the problem \eqref{eq01} has at least two positive solutions.
\end{theorem}

As for the asymptotic behavior of the solutions obtained in Theorem \ref{th01}
 as $\lambda\to0$, we have the following result.

\begin{theorem}\label{th02}
Assume that a sequence $\{\lambda_{n}\}$ satisfies $\lambda_{n}>0$
and $$\lambda_{n}\to 0\,\,\text{as}\,\,n\to \infty.$$
Then there exists a subsequence $\{\lambda_n\}$ and two
sequence $\{u_{n}^{(j)}(x)\}$ $(\,j=1.2\,)$ of positive solutions of
Eq.\eqref{eq01} such that
\begin{itemize}
\item[$(i)$]$\|u_{n}^{(1)}\|_{H^{\frac{\alpha}{2}}_{0}(\Omega)}\to 0$ as $n\to \infty$;
\item[$(ii)$] There exist two sequence $\{x_{n}\}\subset \Omega$, $\{R_{n}\}\subset \mathbb{R}^{+}$ and a
positive solution $u_{0}\in H^{\frac{\alpha}{2}}_{0}(\Omega)$ of critical problem
$$(-\Delta)^{\frac{\alpha}{2}}u=|u|^{2^{*}_{s}-2}u,\quad \text{in}\,\, \mathbb{R}^{N},$$
such that $$R_{n}\to +\infty\,\,\text{as}\,\,n\to +\infty$$
and
$$\|u_{n}^{(2)}(x)-R_{n}^{\frac{N-\alpha}{2}}u_{0}(R_{n}(x-x_{n}))\|_{H^{\frac{\alpha}{2}}_{0}(\Omega)}\to 0\,\,\text{as}\,\, n\to \infty.$$
\end{itemize}
\end{theorem}

The paper is organized as follows. In Section 2, we introduce a variational
setting of the problem and present some preliminary results.
In Section 3, some properties of the fractional operator are discussed. Then we give the
 proof of Theorem \ref{th01}. Finally,  the proof of Theorem \ref{th02}
is given in Section 4.

For convenience we fix some notations. $L^{p}(\Omega)$ $(1<p\leq \infty)$
denotes the usual Sobolev space with norm $|\cdot|_{L^p}$;
$C_{0}(\bar{\Omega})$ denotes the space of continuous real functions in
 $\bar{\Omega}$ vanishing on the boundary $\partial\Omega$;
 $C$ or $C_{i}(i=1,2,\cdot\cdot\cdot,)$ denote any positive constant.

\section{Notation and Preliminaries}
Denote the upper half-space in $\mathbb{R}^{N+1}_{+}$ by
$$\mathbb{R}^{N+1}_{+}=\{z=(x,y)=(x_{1},x_{2},\cdot\cdot\cdot,x_{n},y)\in \mathbb{R}^{N+1}|\,y>0\},$$
the half cylinder standing on a bounded smooth domain
$\Omega\subset \mathbb{R}^{N}$ by
$\mathcal{C}_{\Omega}=\Omega\times(0\,,\,\infty)\subset \mathbb{R}^{N+1}_{+}$
 and its lateral boundary given that
 $\partial_{L}\mathcal{C}_{\Omega}=\partial\Omega\times[0\,,\,\infty)$.

 Let $\{\varphi_{j}\}$ be an orthonormal basis of $L^{2}(\Omega)$ with $|\varphi_{j}|_{L^2}=1$
 forming a spectral decomposition of $-\Delta$ in $\Omega$ with zero
Dirichlet boundary conditions and $\lambda_{j}$ be the corresponding eigenvalues.
Let
$$H^{\frac{\alpha}{2}}_{0}(\Omega)=\{u=\sum\limits_{j=1}^{\infty} a_{j}\varphi_{j}\in L^{2}(\Omega):\|u\|_{H^{\frac{\alpha}{2}}_{0}(\Omega)}=\Big(\sum\limits_{j=1}^{\infty} a_{j}^{2}\lambda^{\frac{\alpha}{2}}\Big)^{\frac{1}{2}}<\infty\}.$$
Define the inner product in $H^{\frac{\alpha}{2}}_{0}(\Omega)$ by
$$\langle u\,,\,v\rangle_{H^{\frac{\alpha}{2}}_{0}(\Omega)}
=\int\limits_{\Omega}(-\Delta)^{\frac{\alpha}{4}}u(-\Delta)^{\frac{\alpha}{4}}v\,dx.$$
It is not difficult to see that $H^{\frac{\alpha}{2}}_{0}(\Omega)$ is a Hilbert space.
For any $u\in H^{\frac{\alpha}{2}}_{0}(\Omega)$, $u=\sum\limits_{j=1}^{\infty}a_{j}\varphi_{j}$
with $a_{j}=\int_{\Omega}u\,\varphi_{j}dx$, the fractional power of the Dirichlet Laplacian
$(-\Delta)^{\frac{\alpha}{2}}$ is defined by
$$(-\Delta)^{\frac{\alpha}{2}}u=
\sum\limits_{j=1}^{\infty}a_{j}\,\lambda_{j}^{\frac{\alpha}{2}}\varphi_{j}.$$
\begin{definition}\label{def01}
We say that $u\in H^{\frac{\alpha}{2}}_{0}(\Omega)$ is a
solution of Eq. \eqref{eq01} such that for every function
$\varphi\in H^{\frac{\alpha}{2}}_{0}(\Omega)$, it holds
$$\int\limits_{\Omega}(-\Delta)^{\frac{\alpha}{4}}u(-\Delta)^{\frac{\alpha}{4}}\varphi dx=\lambda\int\limits_{\Omega}f(x)|u|^{q-2}u\varphi dx
+\int\limits_{\Omega}|u|^{2^{*}_{\alpha}-2}u\varphi dx.$$
\end{definition}

 Associated with problem \eqref{eq01} we consider the energy functional
 $$I(u)=\frac{1}{2}\int\limits_{\Omega}|(-\Delta)^{\frac{\alpha}{4}}u|^{2}dx
 -\frac{\lambda}{q}\int\limits_{\Omega}f(x)|u|^{q}dx
 -\frac{1}{2^{*}_{\alpha}}\int\limits_{\Omega}|u|^{2^{*}_{\alpha}}dx.$$
This functional is well defined in $H^{\frac{\alpha}{2}}_{0}(\Omega)$,
and moreover, the critical points of $I$ correspond to weak solutions of
problem \eqref{eq01}.

We now conclude the main ingredients of a recently developed technique
which can deal with fractional power of the Laplacian.
To treat the nonlocal problem \eqref{eq01}, we will study a
corresponding extension problem, so that we can investigate
problem \eqref{eq01} by studying a local problem via classical
nonlinear variational methods.

We first define the extension operator and fractional Laplacian for
functions in $H^{\frac{\alpha}{2}}_{0}(\Omega)$.

\begin{definition}
Given a function $u\in H^{\frac{\alpha}{2}}_{0}(\Omega)$,
we define its $\alpha$-harmonic extension $w=E_{\alpha}(u)$
to the cylinder $\mathcal{C}_{\Omega}$ as a solution of the problem
\begin{equation*}\left\{
\aligned
&div(y^{1-\alpha}\nabla w)=0,\quad &\text{in}\,\,\mathcal{C}_{\Omega},\\
&w=0,&\text{on}\,\,\partial_{L}\mathcal{C}_{\Omega},\\
&w=u,&\text{on}\,\,\Omega\times\{0\}.
\endaligned\right.
\end{equation*}
\end{definition}
Following \cite{ll2007}, we can define the fractional Laplacian
operator by the Dirichlet to Neumann map as follows.
\begin{definition}
For any regular function $u(x)$, the fractional Laplacian
$(-\Delta)^{\frac{\alpha}{2}}$ acting on $u$ is defined by
$$(-\Delta)^{\frac{\alpha}{2}}u(x)=-\kappa_{\alpha}\lim\limits_{y\to 0^{+}}y^{1-\alpha}\frac{\partial w}{\partial y}(x,y),\quad \forall x\in \Omega,\quad y\in (0\,,\,\infty),$$
where $w=E_{\alpha}(u)$ and $\kappa_{\alpha}$ is a normalization constant.
\end{definition}

Define $H^{\frac{\alpha}{2}}_{0,L}(\mathcal{C}_{\Omega})$
as the closure of $C^{\infty}_{0}(\Omega)$ under the norm
$$\|w\|_{H^{\frac{\alpha}{2}}_{0,L}(\mathcal{C}_{\Omega})}
=\Big(\kappa_{\alpha}\int_{\mathcal{C}_{\Omega}}y^{1-\alpha}|\nabla w|^{2}dxdy\Big)^{\frac{1}{2}}.$$

From \cite{bc2013} and \cite{ll2007}, the map $E_{\alpha}(\cdot)$ is an
isometry between  $H^{\frac{\alpha}{2}}_{0}(\Omega)$ and
$H^{\frac{\alpha}{2}}_{0,L}(\mathcal{C}_{\Omega})$.
Furthermore, we have
\begin{itemize}
\item[(i)]$\|(-\Delta)^{\frac{\alpha}{2}}u\|_{H^{-\frac{\alpha}{2}}(\Omega)}=
\|u\|_{H^{\frac{\alpha}{2}}_{0}(\Omega)}
=\|E_{\alpha}(u)\|_{H^{\frac{\alpha}{2}}_{0,L}(\mathcal{C}_{\Omega})}$,
where $H^{-\frac{\alpha}{2}}(\Omega)$ denotes the dual
space of $H^{\frac{\alpha}{2}}_{0}(\Omega)$;
\item[(ii)]For any $w\in H^{\frac{\alpha}{2}}_{0,L}(\mathcal{C}_{\Omega})$,
there exists a constant $C$ independent of $w$ such that
$$\|\text{tr}_{\Omega}w\|_{L^{r}(\Omega)}\leq C\|w\|_{H^{\frac{\alpha}{2}}_{0,L}(\mathcal{C}_{\Omega})}$$
holds for every $r\in [2\,,\,\frac{2N}{N-\alpha}]$.
Moreover, $H^{\frac{\alpha}{2}}_{0,L}(\mathcal{C}_{\Omega})$
is compactly embedded into
$L^{r}(\Omega)$ for $r\in [2\,,\,\frac{2N}{N-\alpha})$.
\end{itemize}

Now we can transform the nonlocal problem \eqref{eq01} into
the following local problem:
\begin{equation}\label{eq02}
\left\{\aligned
&-div(y^{1-\alpha}\nabla w)=0,\quad &\text{in}\,\,\mathcal{C}_{\Omega},\\
&w=0,       &\text{on}\,\,\partial_{L}\mathcal{C}_{\Omega},\\
&\frac{\partial w}{\partial \upsilon^{\alpha}}=\lambda f(x)|w|^{q-2}w+|w|^{2^*_{\alpha}-2}w    &\text{on}\,\,\Omega\times\{0\},\\
\endaligned\right.
\end{equation}
where $\frac{\partial w}{\partial\upsilon^{\alpha}}:=-\kappa_{\alpha}\lim\limits_{y\to 0^{+}}y^{1-\alpha}
\frac{\partial w}{\partial y}(x,y)$, $\forall x\in \Omega$.
 An energy solution to this problem is a function
 $w\in H^{\frac{\alpha}{2}}_{0,L}(\mathcal{C}_{\Omega})$
 such that
 $$\kappa_{\alpha}\int\limits_{\mathcal{C}_{\Omega}}y^{1-\alpha}\langle \nabla w\,,\,\nabla \varphi\rangle dxdy
 =\lambda\int\limits_{\Omega\times\{0\}}f(x)|w|^{q-2}w \,\varphi dx
 +\int\limits_{\Omega\times\{0\}}|w|^{2^{*}_{\alpha}-2}w\,\varphi dx$$
for all $\varphi\in H^{\frac{\alpha}{2}}_{0,L}(\mathcal{C}_{\Omega})$.

If $w$ satisfies \eqref{eq02}, then the trace
$u=\text{tr}_{\Omega}w=w(x,0)\in H^{\frac{\alpha}{2}}_{0}(\Omega)$
is an energy solution to problem \eqref{eq01}. The converse is also true.
By the equivalence of these two formulations, we will use both formulations
in the sequel to their best advantage.

The associated energy functional to problem \eqref{eq02} is
\begin{equation*}\label{eq03}
J(w)=\frac{\kappa_{\alpha}}{2}\int\limits_{\mathcal{C}_{\Omega}}y^{1-\alpha}|\nabla w|^{2}dxdy
-\frac{\lambda}{q}\int\limits_{\Omega\times\{0\}}f(x)|w|^{q}dx
-\frac{1}{2^{*}_{\alpha}}\int\limits_{\Omega\times\{0\}}|w|^{2^*_{\alpha}}dx,
\end{equation*}
for all $w\in H^{\frac{\alpha}{2}}_{0,L}(\mathcal{C}_{\Omega})$.
Clearly, the critical points of $J$ in $H^{\frac{\alpha}{2}}_{0,L}(\mathcal{C}_{\Omega})$
correspond to critical points of $I$ in $H^{\frac{\alpha}{2}}_{0}(\Omega)$.

In the following lemma, we will list some inequalities.
\begin{lemma}\label{le}
For every $1\leq r\leq \frac{2N}{N-\alpha}$, and every
$w\in H^{\frac{\alpha}{2}}_{0,L}(\mathcal{C}_{\Omega})$, it holds
$$\Big(\int\limits_{\Omega\times\{0\}}|w|^{r}dx\Big)^{\frac{2}{r}}
\leq C\,\kappa_{\alpha}\int\limits_{\mathcal{C}_{\Omega}}y^{1-\alpha}|\nabla w|^{2}dxdy,$$
where the constant $C$ depends on $r,\,\alpha,\,N,\,|\Omega|$.
\end{lemma}

\begin{lemma}\label{le1}
For every $w\in H^{\frac{\alpha}{2}}(\mathbb{R}^{N+1}_{+})$, it holds
\begin{equation}\label{eq04}
S(\alpha,N)\Big(\int\limits_{\mathbb{R}^{N}}|u|^{\frac{2N}{N-\alpha}}\Big)^{\frac{N-\alpha}{N}}\leq \int\limits_{\mathbb{R}^{N+1}_{+}}y^{1-\alpha}|\nabla w|^{2}dxdy,
\end{equation}
where $u=\text{tr}_{\Omega}\,w$. The best constant takes the exact value
$$S(\alpha,N)=\frac{2\,\pi^{\frac{\alpha}{2}}\Gamma(\frac{2-\alpha}{2})\Gamma(\frac{N+\alpha}{2})(\Gamma(\frac{N}{2}))^{\frac{\alpha}{N}}       }{\Gamma(\frac{\alpha}{2})\Gamma(\frac{N-\alpha}{2})(\Gamma(N))^{\frac{\alpha}{N}}}$$
and can be achieved when $u(x)=w(x,0)$ takes the form
\begin{equation}\label{eq05}
u_{\varepsilon}(x)=\frac{\varepsilon^{\frac{N-\alpha}{2}}}{(\varepsilon^{2}+|x|^{2})^{\frac{N-\alpha}{2}}}
\end{equation}
for $\varepsilon>0$ arbitrary and $w=E_{\alpha}(u)$.
\end{lemma}

Now we are looking for the solutions of problem \eqref{eq01}.
Equivalently, we consider the solutions of problem \eqref{eq02}.
First we consider the Nehari minimization problem, i.e., for $\lambda>0$,
$$m_{J}=\inf\{J(w)|\,w\in \mathcal{N}\},$$
where $$\mathcal{N}=\{w\in H^{\frac{\alpha}{2}}_{0,L}(\mathcal{C}_{\Omega})\,|\,\,\langle J'(w)\,,\,w\rangle=0\}.$$

 Define
 $$\Psi(w)=\langle J'(w)\,,\,w\rangle=\kappa_{\alpha}\int\limits_{\mathcal{C}_{\Omega}}y^{1-\alpha}|\nabla w|^{2}dxdy-\lambda\, \int\limits_{\Omega\times\{0\}} f(x)|w|^{q}dx-\int\limits_{\Omega\times\{0\}} |w|^{2^*_{\alpha}}dx.$$
 Then, for $w\in \mathcal{N}$,
 $$\langle\Psi'(w)\,,\,w\rangle=2\kappa_{\alpha}\int\limits_{\mathcal{C}_{\Omega}}y^{1-\alpha}|\nabla w|^{2}dxdy-\lambda\, q \int\limits_{\Omega\times\{0\}} f(x)|w|^{q}dx-2^*_{\alpha}\int\limits_{\Omega\times\{0\}} |w|^{2^*_{\alpha}}dx.$$
 Similar to the method used in \cite{wu2008,jx2013},
  we split $\mathcal{N}$ into three parts:
\begin{equation*}
\aligned
&\mathcal{N}^{+}=\{w\in \mathcal{N}\,|\,\langle\Psi'(w)\,,\,w\rangle>0\};\\
 &\mathcal{N}^{0}=\{w\in \mathcal{N}\,|\,\langle\Psi'(w)\,,\,w\rangle=0\};\\
&\mathcal{N}^{-}=\{w\in\mathcal{N}\,|\,\langle\Psi'(w)\,,\,w\rangle<0\}.
\endaligned
\end{equation*}
Then we have the following results.
\begin{lemma}\label{le0}
Let $\lambda_{1}=\Big(\frac{2^{*}_{\alpha}-2}{2^{*}_{\alpha}-q}\Big)
\Big(\frac{2-q}{2^{*}_{\alpha}-q}\Big)^{\frac{2-q}{2^{*}_{\alpha}-q}}
\Big(\kappa_{\alpha}S(\alpha,N)\Big)^{\frac{2^{*}_{\alpha}-q}{2^{*}_{\alpha}-2}}\,|f|_{\infty}^{-1}$.
Then for every $w\in H^{\frac{\alpha}{2}}_{0,L}(\mathcal{C}_{\Omega})$,
$w\neq 0$ and $\lambda\in(0\,,\,\lambda_{1}) $,
there exist unique $t^{+}(w)$ and $t^{-}(w)$ such that
\begin{itemize}
\item[$(1)$]$0\leq t^{+}(w)<t_{max}=
    \Big(\frac{(2-q)\kappa_{\alpha}\int\limits_{\mathcal{C}_{\Omega}}y^{1-\alpha}|\nabla w|^{2}dxdy}
    {(2^{*}_{\alpha}-q)\int\limits_{\Omega\times\{0\}}|w|^{2^{*}_{\alpha}}dx}\Big)^{\frac{1}{2^{*}_{\alpha}-2}}
    <t^{-}(w)$;
\item[$(2)$]$t^{-}(w)w\in \mathcal{N}^{-}$ and $t^{+}(w)w\in \mathcal{N}^{+}$;
\item[$(3)$] $\mathcal{N}^{-}=\Big{\{}w\in H^{\frac{\alpha}{2}}_{0,L}(\mathcal{C}_{\Omega})\setminus\{0\}:\,
t^{-}(\frac{w}{\|w\|_{H^{\frac{\alpha}{2}}_{0,L}(\mathcal{C}_{\Omega})}})
=\|w\|_{H^{\frac{\alpha}{2}}_{0,L}(\mathcal{C}_{\Omega})}
\Big{\}}$;
\item[$(4)$] $J(t^{-}w)=\max\limits_{t\geq t_{max}}J(tw)$
and $J(t^{+}w)=\min\limits_{t\in [0, t^{-}]}J(tw)$.
\end{itemize}
Moreover, $t^{+}(w)>0$ if and only if $\int\limits_{\Omega\times\{0\}}f(x)|w|^{q}dx>0$.
\end{lemma}

\begin{proof}
The proof is almost the same as that in \cite{wu2008}.
We need only to define
$$s(t)=t^{2-q}\kappa_{\alpha}\int\limits_{\mathcal{C}_{\Omega}}y^{1-\alpha}|\nabla w|^{2}dxdy-t^{2^{*}_{\alpha}-q}\int\limits_{\Omega\times\{0\}}|w|^{2^{*}_{\alpha}}dx.$$
Thus, we omit the details here.
\end{proof}

\begin{lemma}\label{le3}
There exists $\lambda_{2}>0$ such that for each $\lambda\in (0\,,\,\lambda_{2})$,
we have $\mathcal{N}^{0}=\{0\}$.
\end{lemma}
\begin{proof}
Suppose the contrary, there exists a $w\in \mathcal{N}^{0}\setminus\{0\}$,
such that
\begin{equation}\label{eq-le3}
\langle\Psi'(w)\,,\,w\rangle=0.
\end{equation}
Then, we consider the following two cases.

Case (i): $\int\limits_{\Omega\times\{0\}} f(x)|w|^{q}dx=0$.
 Then
\begin{equation*}\aligned
\langle \Psi'(w)\,,\,w\rangle&=2\kappa_{\alpha}\int\limits_{\mathcal{C}_{\Omega}}y^{1-\alpha}|\nabla w|^{2}dxdy-\lambda\, q \int\limits_{\Omega\times\{0\}} f(x)|w|^{q}dx
-\frac{2N}{N-\alpha}\int\limits_{\Omega\times\{0\}} |w|^{2^*_{\alpha}}dx\\
&=2\kappa_{\alpha}\int\limits_{\mathcal{C}_{\Omega}}y^{1-\alpha}|\nabla w|^{2}dxdy
-\frac{2N}{N-\alpha}\kappa_{\alpha}\int\limits_{\mathcal{C}_{\Omega}}y^{1-\alpha}|\nabla w|^{2}dxdy\\
&=-\frac{2\alpha}{N-\alpha}\|w\|^{2}_{H^{\frac{\alpha}{2}}_{0,L}(\mathcal{C}_{\Omega})}<0.
\endaligned
\end{equation*}
So, in this case $w\in \mathcal{N}^{-}$.

Case (ii): $\int\limits_{\Omega\times\{0\}} f(x)|w|^{q}dx\neq 0$. From \eqref{eq-le3}, we get  that
\begin{equation*}\label{eq07}\aligned
0&=2\kappa_{\alpha}\int\limits_{\mathcal{C}_{\Omega}}y^{1-\alpha}|\nabla w|^{2}dxdy-\lambda\, q \int\limits_{\Omega\times\{0\}} f(x)|w|^{q}dx
-2^*_{\alpha}\int\limits_{\Omega\times\{0\}} |w|^{2^*_{\alpha}}dx\\
&=(2-q)\kappa_{\alpha}\int\limits_{\mathcal{C}_{\Omega}}y^{1-\alpha}|\nabla w|^{2}dxdy-(2^{*}_{\alpha}-q)
\int\limits_{\Omega\times\{0\}} |w|^{2^*_{\alpha}}dx.\\
\endaligned\end{equation*}
This implies that
\begin{equation}\label{eq08}
\|w\|^{2}_{H^{\frac{\alpha}{2}}_{0,L}(\mathcal{C}_{\Omega})}
=\frac{2^*_{\alpha}-q}{2-q}\int\limits_{\Omega\times\{0\}} |w|^{2^*_{\alpha}}dx.
\end{equation}
Moreover, we have
\begin{equation}\label{eq09}
\aligned
\lambda\int\limits_{\Omega\times\{0\}}f(x)|w|^{q}dx
&=\kappa_{\alpha}\int\limits_{\mathcal{C}_{\Omega}}y^{1-\alpha}|\nabla w|^{2}dxdy-
\int\limits_{\Omega\times\{0\}} |w|^{2^*_{\alpha}}dx\\
&=\kappa_{\alpha}\int\limits_{\mathcal{C}_{\Omega}}y^{1-\alpha}|\nabla w|^{2}dxdy-
\frac{2-q}{2^*_{\alpha}-q}\kappa_{\alpha}\int\limits_{\mathcal{C}_{\Omega}}y^{1-\alpha}|\nabla w|^{2}dxdy\\
&=\frac{2^{*}_{\alpha}-2}{2^{*}_{\alpha}-q}\|w\|^{2}_{H^{\frac{\alpha}{2}}_{0,L}(\mathcal{C}_{\Omega})}.
\endaligned\end{equation}
Then, by \eqref{eq09}, \eqref{eq04} and the H\"{o}lder inequality, we obtain
\begin{equation}\label{eq09-1}
\|w\|^{2-q}_{H^{\frac{\alpha}{2}}_{0,L}(\mathcal{C}_{\Omega})}\leq \lambda\Big(\frac{2^{*}_{\alpha}-q}{2^{*}_{\alpha}-2}\Big)
\,\Big(\kappa_{\alpha}S(\alpha,N)\Big)^{-\frac{q}{2}}|f|_{L^{\infty}}.
\end{equation}

Let $K:\mathcal{N}\to \mathbb{R}$ be given by
$$ K(w)=C(N,\alpha)
\|w\|^{\frac{2(2^*_{\alpha}-1)}{2^*_{\alpha}-2}}_{H^{\frac{\alpha}{2}}_{0,L}(\mathcal{C}_{\Omega})}\Big(\int\limits_{\Omega\times\{0\}} |w|^{2^*_{\alpha}}dx\Big)^{\frac{1}{2-2^*_{\alpha}}}-\lambda\int\limits_{\Omega\times\{0\}}f(x)|w|^{q}dx,$$
where $C(N,\alpha)=\Big{(}\frac{2^{*}_{\alpha}-2}{2-q}\Big{)}
\Big{(}\frac{2-q}{2^{*}_{\alpha}-q}\Big{)}^{\frac{2^*_{\alpha}-1}{2^*_{\alpha}-2}}$.
Then $K(w)=0$
for all $w\in \mathcal{N}^{0}$. Indeed, by \eqref{eq08} and \eqref{eq09},
\begin{equation*}
\aligned
 K(w)
 &=C(N,\alpha)
\|w\|^{\frac{2(2^*_{\alpha}-1)}{2^*_{\alpha}-2}}_{H^{\frac{\alpha}{2}}_{0,L}(\mathcal{C}_{\Omega})}\Big(\int\limits_{\Omega\times\{0\}} |w|^{2^*_{\alpha}}dx\Big)^{\frac{1}{2-2^*_{\alpha}}}-\lambda\int\limits_{\Omega\times\{0\}}f(x)|w|^{q}dx\\
&=\Big{(}\frac{2^{*}_{\alpha}-2}{2-q}\Big{)}\Big{(}\frac{2-q}{2^{*}_{\alpha}-q}\Big{)}^{\frac{2^*_{\alpha}-1}{2^*_{\alpha}-2}}
\Big(\frac{2^*_{\alpha}-q}{2-q}\Big)^{\frac{1}{2^*-2}} \|w\|^{2}_{H^{\frac{\alpha}{2}}_{0,L}(\mathcal{C}_{\Omega})}-\lambda\int\limits_{\Omega\times\{0\}}f(x)|w|^{q}dx\\
&=\frac{2^*_{\alpha}-2}{2^*_{\alpha}-q}\,\|w\|^{2}_{H^{\frac{\alpha}{2}}_{0,L}(\mathcal{C}_{\Omega})}
-\lambda\int\limits_{\Omega\times\{0\}}f(x)|w|^{q}dx\\
&=0.
\endaligned
\end{equation*}

On the other hand, by \eqref{eq09},\eqref{eq09-1}, we have
\begin{equation}\label{eq00}
\aligned
K(w)
&\geq C(N,\alpha)
\|w\|^{\frac{2(2^*_{\alpha}-1)}{2^*_{\alpha}-2}}_{H^{\frac{\alpha}{2}}_{0,L}(\mathcal{C}_{\Omega})}\Big{(}\int\limits_{\Omega\times\{0\}} |w|^{2^*_{\alpha}}dx\Big{)}^{\frac{1}{2-2^*_{\alpha}}}-\lambda |f|_{L^{\infty}}|w|_{L^{2^*_{\alpha}}}^{q}\\
&\geq C(N,\alpha)\Big(\kappa_{\alpha}\,S(\alpha,N)\Big)^{\frac{N+\alpha}{2\alpha}}
\Big(\int\limits_{\Omega\times\{0\}}|w|^{2^*_{\alpha}}dx\Big)^{\frac{1}{2^*_{\alpha}}}
-\lambda|f|_{L^{\infty}}|w|_{L^{2^*_{\alpha}}}^{q} \\
&\geq \lambda|f|_{L^{\infty}}|w|^{q}_{L^{2^*}}\Big{[}
\frac{\Big(\frac{2^*_{\alpha}-2}{2^*_{\alpha}-q}\Big)^{\frac{1}{2-q}}
\Big(2-q\Big)^{\frac{n-\alpha}{2\alpha}}
\Big(\kappa_{\alpha}\,S(\alpha,N)\Big)^{\frac{q-1}{2-q}+\frac{N+\alpha}{2\alpha}}}
{(\lambda|f|_{L^{\infty}})^{\frac{1}{2-q}}}
-1 \Big{]}.\\
\endaligned
\end{equation}
This implies that there exists
$$\lambda_{2}:=|f|_{L^{\infty}}^{-1}\Big(\frac{2^*_{\alpha}-2}{2^*_{\alpha}-q}\Big)
\Big(2-q\Big)^{\frac{(N-\alpha)(2-q)}{2\alpha}}
(\kappa_{\alpha}\,S(\alpha,N))^{\frac{2\alpha(q-1)+(N+\alpha)(2-q)}{2\alpha}}
$$
 such that for each
$\lambda\in (0\,,\,\lambda_{2})$, we have $K(w)>0$ for all
$w\in \mathcal{N}^{0}\setminus\{0\}$, which yields
a contradiction. Thus, we can conclude that $\mathcal{N}^{0}=\{0\}$
for all $\lambda\in (0\,,\,\lambda_{2})$.
\end{proof}

\begin{lemma}\label{le4}
If $w\in \mathcal{N}^{+}$ and $w\not\neq0$, then $\int\limits_{\Omega\times\{0\}}f(x)|w|^{q}dx>0$.
\end{lemma}

\begin{proof}
 From $w\in \mathcal{N}^{+}$, we have
\begin{equation*}\aligned
2\kappa_{\alpha}\int\limits_{\mathcal{C}_{\Omega}}y^{1-\alpha}|\nabla w|^{2}dxdy&> \lambda q\int\limits_{\Omega\times\{0\}} f(x)|w|^{q}dx
+2^{*}_{\alpha}\int\limits_{\Omega\times\{0\}} |w|^{2^*_{\alpha}}dx\\
&=q\kappa_{\alpha}\int\limits_{\mathcal{C}_{\Omega}}y^{1-\alpha}|\nabla w|^{2}dxdy+(2^*_{\alpha}-q)\int\limits_{\Omega\times\{0\}} |w|^{2^*_{\alpha}}dx,\\
\endaligned
\end{equation*}
that is, $$\kappa_{\alpha}\int\limits_{\mathcal{C}_{\Omega}}y^{1-\alpha}|\nabla w|^{2}dxdy> \frac{2^{*}_{\alpha}-q}{2-q}\int\limits_{\Omega\times\{0\}} |w|^{2^*_{\alpha}}dx.$$
Then, we have
\begin{equation*}\aligned
\lambda\int\limits_{\Omega\times\{0\}} f(x)|u|^{q+1}dx
&=\kappa_{\alpha}\int\limits_{\mathcal{C}_{\Omega}}y^{1-\alpha}|\nabla w|^{2}dxdy-\int\limits_{\Omega\times\{0\}}|w|^{2^*_{\alpha}}dx\\
&>\frac{2^*_{\alpha}-2}{2-q}\int\limits_{\Omega\times\{0\}}|w|^{2^{*}_{\alpha}}dx>0.\\
\endaligned\end{equation*}
This completes the proof.
\end{proof}

The following lemma shows that the minimizers on $\mathcal{N}$ are
actually the critical points of $J$.
\begin{lemma}\label{le5}
For $\lambda\in (0\,,\,\lambda_{2})$.
If $w\in H^{\frac{\alpha}{2}}_{0,L}(\mathcal{C}_{\Omega})$ is a local
minimizer for $J$ on $\mathcal{N}$, then
$J'(w)=0$ in $H^{-\frac{\alpha}{2}}(\mathcal{C}_{\Omega})$,
where $H^{-\frac{\alpha}{2}}(\mathcal{C}_{\Omega})$ denotes the dual space of $H^{\frac{\alpha}{2}}_{0,L}(\mathcal{C}_{\Omega})$.
\end{lemma}
\begin{proof}
If $w_{0}\in \mathcal{N}$ is a local minimizer of $J$, then $w_{0}$ is
a nontrivial solution of the optimization problem
\begin{equation*}
\text{minimize} \,\,J(w)\,\,
\text{subject to}\,\,\langle \Psi'(w)=0,
\end{equation*}
Hence by the theory of Lagrange multiplies, there exists $\theta\in \mathbb{R}$
such that
$J'(w_{0})=\theta\Psi'(w_{0})$ in $H^{-\frac{\alpha}{2}}(\mathcal{C}_{\Omega})$.
This implies that
\begin{equation}\label{eq11}
\langle J'(w_{0})\,,\,w_{0}\rangle=\theta\langle \Psi'(w_{0})\,,\,w_{0}\rangle.
\end{equation}

By Lemma \ref{le3}, for every $w\not\neq0$,
we have $\langle\Psi'(w_{0})\,,\,w_{0}\rangle\neq0$ and so by \eqref{eq11}, $\theta=0$.
This completes the proof.
\end{proof}

\begin{lemma}\label{le6}
The functional $J$ is coercive and bounded from below on $\mathcal{N}$.
\end{lemma}
\begin{proof}
For $w\in \mathcal{N}$, we have
\begin{equation*}\aligned
J(w)
&=(\frac{1}{2}-\frac{1}{2^*_{\alpha}})\kappa_{\alpha}\int\limits_{\mathcal{C}_{\Omega}}y^{1-\alpha}|\nabla w|^{2}dxdy-\lambda(\frac{1}{q}-\frac{1}{2^*_{\alpha}})\int\limits_{\Omega\times\{0\}} f(x)|u|^{q}dx\\
&\geq \frac{\alpha}{2N}\|w\|^{2}_{H^{\frac{\alpha}{2}}_{0,L}(\mathcal{C}_{\Omega})}-
\lambda\Big(\frac{2^{*}_{\alpha}-q}{q2^*_{\alpha}}\Big) |f|_{L^{\infty}}\Big(\kappa_{\alpha}
\,S(\alpha,N)\Big)^{-\frac{q}{2}}\|w\|^{q}_{H^{\frac{\alpha}{2}}_{0,L}(\mathcal{C}_{\Omega})}\\
&\geq \frac{q-2}{2}\Big(\frac{N}{\alpha}\Big)^{\frac{q}{2-q}}\Big(\lambda\,C\Big)^{\frac{2}{2-q}},\\
\endaligned\end{equation*}
where $C=\Big(\frac{2^{*}_{\alpha}-q}{q2^*_{\alpha}}\Big) |f|_{L^{\infty}}\kappa_{\alpha}^{-\frac{q}{2}}
\,S(\alpha,N)^{-\frac{q}{2}}$.
This tell us that $J$ is coercive and bounded from below on $\mathcal{N}$.
\end{proof}

In the end of this section, we will use the idea of \cite{ct1992}
 to get the property of $\mathcal{N}$.
\begin{lemma}\label{le7}
For each $w\in \mathcal{N}$, $w\not\equiv0$,
 there exists $r>0$ and a differentiable function $t=t(v)$
such that $t=t(v)>0$ for all $v\in \{w\in H^{\frac{\alpha}{2}}_{0,L}(\mathcal{C}_{\Omega}):\,
\|w\|^{2}_{H^{\frac{\alpha}{2}}_{0,L}(\mathcal{C}_{\Omega})}<\varepsilon\}$
satisfying that
$$ t(0)=1,\, t(v)(w-v)\in \mathcal{N},$$
and
$$\langle t'(0)\,,\,v\rangle
=\frac{2\kappa_{\alpha}\int\limits_{\mathcal{C}_{\Omega}} y^{1-\alpha}\nabla w\nabla vdxdy
-q\lambda\int\limits_{\Omega\times\{0\}} f(x)|w|^{q-2}wvdx
-2^*_{\alpha}\int\limits_{\Omega\times\{0\}}  |w|^{2^*_{\alpha}-2}wvdx}
{(2-q)\kappa_{\alpha}\int\limits_{\mathcal{C}_{\Omega}} y^{1-\alpha}|\nabla w|^{2}dxdy-(2^*_{\alpha}-q)\int\limits_{\Omega\times\{0\}}  |w|^{2^*_{\alpha}}dx}$$
for all $v\in H^{\frac{\alpha}{2}}_{0,L}(\mathcal{C}_{\Omega})$.
\end{lemma}

\begin{proof}
 Define $F:\,\mathbb{R}\times H^{\frac{\alpha}{2}}_{0,L}(\mathcal{C}_{\Omega})\to \mathbb{R}$ as follows.
\begin{equation*}\label{eq27}
\aligned
F(t,v)
&=\langle J'(t(w-v))\,,\,t(w-v)\rangle\\
&=t^{2}\kappa_{\alpha}\int\limits_{\mathcal{C}_{\Omega}}y^{1-\alpha} |\nabla (w-v)|^{2}dxdy-\lambda\,t^{q}\int\limits_{\Omega\times\{0\}}  f(x)|w-v|^{q}dx-t^{2^*_{\alpha}}\int\limits_{\Omega\times\{0\}}  |w-v|^{2^*_{\alpha}}dx,
\endaligned
\end{equation*}
for all $v\in \mathcal{N}$.

Since $F(1,0)=\langle J'(w)\,,\,w\rangle=0$ and by Lemma \ref{le3},
we obtain
\begin{equation*}\label{eq28}
\aligned
F'_{t}(1,0)
&=2\kappa_{\alpha}\int\limits_{\mathcal{C}_{\Omega}} y^{1-\alpha}|\nabla w|^{2}dxdy
-\lambda q\int\limits_{\Omega\times\{0\}}  f(x)|w|^{q}dx
-2^*_{\alpha}\int\limits_{\Omega\times\{0\}}  |w|^{2^*_{\alpha}}dx\\
&=(2-q)\kappa_{\alpha}\int\limits_{\mathcal{C}_{\Omega}}y^{1-\alpha}|\nabla w|^{2}dxdy-(2^*_{\alpha}-q)\int\limits_{\Omega\times\{0\}}  |w|^{2^*_{\alpha}}dx\neq0.
\endaligned
\end{equation*}
Applying the implicit function theorem at the point $(1,0)$,
we get that there exist $\varepsilon>0$ small and a function
$t=t(v)$ satisfying $t(0)=1$ and
\begin{equation*}
\langle t'(0)\,,\,v\rangle
=\frac{2\kappa_{\alpha}\int\limits_{\mathcal{C}_{\Omega}} y^{1-\alpha}\nabla w\nabla vdxdy
-q\lambda\int\limits_{\Omega\times\{0\}}  f(x)|w|^{q-2}wvdx
-2^{*}_{\alpha}\int\limits_{\Omega\times\{0\}}  |w|^{2^*_{\alpha}-2}wvdx}
{(2-q)\kappa_{\alpha}\int\limits_{\mathcal{C}_{\Omega}}y^{1-\alpha} |\nabla w|^{2}dxdy-(2^*_{\alpha}-q)\int\limits_{\Omega\times\{0\}}  |w|^{2^*_{\alpha}dx}}.
\end{equation*}
Moreover, there is a $t(v)$ such that $F(t(v),\,v)=0 $ for all $v\in \{w\in H^{\frac{\alpha}{2}}_{0,L}(\mathcal{C}_{\Omega}):\,
\|w\|^{2}_{H^{\frac{\alpha}{2}}_{0,L}(\mathcal{C}_{\Omega})}<\varepsilon\}$,
which is equivalent to $\langle J'(t(v)(w-v))\,,\,t(v)(w-v)\rangle=0$,
that is, $t(v)(w-v)\in \mathcal{N}$. We prove the lemma.
\end{proof}

\section{Proof of Theorem \ref{th01}}
\subsection{The minimizer solution on $\mathcal{N}^{+}$}
Let \begin{equation}\label{lambda}
\lambda_{*}=\min\{\lambda_{1}\,,\,\lambda_{2}\}.
\end{equation}

In this subsection, we show that problem \eqref{eq02} has a
position solution if $\lambda<\lambda_{*}$, which is the minimizer
of $J$ on $\mathcal{N}^{+}$.

By Lemma \ref{le6}, for $\lambda\in (0\,,\,\lambda_{*})$,
$J$ is coercive and bounded from below on $\mathcal{N}$ and so on $\mathcal{N}^{+}$.
Therefore, we define
$$m^{+}=\inf \{J(w):\,w\in\mathcal{N}^{+}\}. $$

Now we consider the following auxiliary equation:
\begin{equation}\label{eq20}
\left\{\aligned
&-div(y^{1-\alpha}\nabla w)=0,\quad &\text{in}\,\,\mathcal{C}_{\Omega},\\
&w=0,       &\text{on}\,\,\partial_{L}\mathcal{C}_{\Omega},\\
&\frac{\partial w}{\partial \upsilon^{\alpha}}=\lambda f(x)|w|^{q-2}w  &\text{on}\,\,\Omega\times\{0\},\\
\endaligned\right.
\end{equation}
In this case we use the notation $F$ and $\mathcal{M}$ respectively,
for the energy functional and the natural
constrain, namely,
\begin{equation*}\label{eq21}
F(w)=\frac{\kappa_{\alpha}}{2}\int\limits_{\mathcal{C}_{\Omega}}y^{1-\alpha}|\nabla w|^{2}dxdy-\frac{\lambda}{q}\int\limits_{\Omega\times\{0\}} f(x)|w|^{q}dx.
\end{equation*}
$$\mathcal{M}=\{w\in H^{\frac{\alpha}{2}}_{0,L}(\mathcal{C}_{\Omega})\setminus\{0\}:\,\langle F'(w)\,,\,w\rangle=0\}.$$
Setting $$m_{F}=\inf\{F(w):\,w\in \mathcal{M}\},$$
then we have the following result.
\begin{theorem}\label{t3}
For each $\lambda>0$, problem \eqref{eq20}
has a positive solution $w_{0}$ such that
$$F(w_{0})=m_{\lambda}<0.$$
\end{theorem}

\begin{proof}
We start by showing that $F$ is coercive,
bounded from below on $\mathcal{M}$ and $m_{\lambda}<0$.
Indeed, for any $w\in \mathcal{M}$, we have
\begin{equation}\label{eq22}
\kappa_{\alpha}\int\limits_{\mathcal{C}_{\Omega}}y^{1-\alpha}|w|^{2}dxdy
=\lambda\int\limits_{\Omega\times\{0\}} f(x)|w|^{q}dx\leq \lambda|f|_{L^{\infty}}\,
(\kappa_{\alpha}S(\alpha,N)^{-\frac{q}{2}})\|w\|^{q}_{H^{\frac{\alpha}{2}}_{0,L}(\mathcal{C}_{\Omega})}.
\end{equation}
This implies
$$F(w)\geq \frac{1}{2}\|w\|^{2}_{H^{\frac{\alpha}{2}}_{0,L}(\mathcal{C}_{\Omega})}
-\frac{1}{q}\lambda\,|f|_{L^{\infty}}\,\Big(\kappa_{\alpha}S(\alpha,N)\Big)
^{-\frac{q}{2}}\|w\|^{q}_{H^{\frac{\alpha}{2}}_{0,L}(\mathcal{C}_{\Omega})}, $$
and therefore, we easily derive the coerciveness for $1<q<2$.
Moreover, \eqref{eq22} implies
\begin{equation}\label{eq23}
\|w\|_{H^{\frac{\alpha}{2}}_{0,L}(\mathcal{C}_{\Omega})}\leq \Big{(}\lambda|f|_{L^{\infty}}\,(\kappa_{\alpha}S(\alpha,N))^{-\frac{q}{2}}\Big{)}^{\frac{1}{2-q}}.
\end{equation}
Hence, for all $w\in \mathcal{M}$ we have $$F(w)=(\frac{1}{2}-\frac{1}{q})\|w\|^{2}_{H^{\frac{\alpha}{2}}_{0,L}(\mathcal{C}_{\Omega})}
\geq -\frac{2-q}{2q}\Big{(}\lambda|f|_{L^s}\,(\kappa_{\alpha}S(\alpha,N))^{-\frac{q}{2}}\Big{)}^{\frac{2}{2-q}}.$$
So $F$ is bounded from below on $\mathcal{M}$
and $m_{\lambda}<0$.

Let $\{w_{n}\}_{n}$ be a minimizing sequence of $F$ on $\mathcal{M}$.
Then, by \eqref{eq23} and the compact imbedding theorem,
there exists a subsequence of $\{w_{n}\}_{n}$, still denoted by $\{w_{n}\}_{n}$, and $w_{0}$ such that
\begin{equation}\label{eq24}\aligned
&w_{n}\rightharpoonup w_{0} \,\,\text{weakly in}\,\,H^{\frac{\alpha}{2}}_{0,L}(\mathcal{C}_{\Omega});\\
&w_{n}(\cdot,0)\to w_{0}(\cdot,0) \,\,\text{strongly in}\,\,L^{p}(\Omega)\,\text{for}\,\, 1<p<2^*_{\alpha};\\
&w_{n}(\cdot,0)\to w_{0}(\cdot,0)\,\,\text{a.e.\,in}\,\Omega.
\endaligned
\end{equation}
Now, we claim that $\int\limits_{\Omega\times\{0\}} f(x)|w_{0}|^{q}dx>0$.
If not, by \eqref{eq24} we obtain
$$\int\limits_{\Omega\times\{0\}}  f(x)|w_{0}|^{q}dx=0$$
and
$$\int\limits_{\Omega\times\{0\}}  f(x)|w_{n}|^{q}dx\to 0\,\,\text{as}\,\,n\to \infty.$$
Hence $\int\limits_{\mathcal{C}_{\Omega}}y^{1-\alpha} |\nabla w_{n}|^{2}dxdy\to 0$
and $F(w_{n})\to 0$ as $n\to \infty$ which is
contradicts $F(w_{n})\to m_{\lambda}<0$ as $n\to \infty$.
Therefore, we have $\int\limits_{\Omega\times\{0\}}  f(x)|w_{0}|^{q}dx>0$.
In particular $w_{0}\not\equiv 0$.

Next, we prove $w_{n}\to w_{0}$ ($n\to\infty$) strongly in $H^{\frac{\alpha}{2}}_{0,L}(\mathcal{C}_{\Omega})$.
Let us suppose on the contrary that
$$\|w_{0}\|_{H^{\frac{\alpha}{2}}_{0,L}(\mathcal{C}_{\Omega})}<\liminf\limits_{n\to \infty}\|w_{n}\|_{H^{\frac{\alpha}{2}}_{0,L}(\mathcal{C}_{\Omega})}\quad \text{as}\,\,n\to \infty$$
and
$$\int\limits_{\Omega\times\{0\}} f(x)|w_{n}|^{q}dx\to \int\limits_{\Omega\times\{0\}} f(x)|w_{0}|^{q}dx\quad \text{as}\,\, n\to \infty.$$ So
\begin{equation}\label{eq1*}
\|w_{0}\|^{2}_{H^{\frac{\alpha}{2}}_{0,L}(\mathcal{C}_{\Omega})}-\lambda\int\limits_{\Omega\times\{0\}} f(x)|w_{0}|^{q}dx<\liminf\limits_{n\to\infty}\Big(\|w_{n}\|^{2}_{H^{\frac{\alpha}{2}}_{0,L}(\mathcal{C}_{\Omega})}
-\lambda\int\limits_{\Omega\times\{0\}} f(x)|w_{n}|^{q}dx\Big)=0.
\end{equation}
From $\int\limits_{\Omega\times\{0\}} f(x)|w_{0}|^{q}dx>0$
and \eqref{eq1*}, we known that the function
$$F(t\,w_{0})
=\frac{t^2}{2}\kappa_{\alpha}\int\limits_{\mathcal{C}_{\Omega}}y^{1-\alpha}|\nabla v_{0}|^{2}dxdy
-\frac{\lambda\, t^{q}}{q}\int\limits_{\Omega\times\{0\}} f(x)|w_{0}|^{q}dx$$
is initially decreasing and eventually
increasing on $t$ with a single turning point $t_{0}\neq1$
such that $t_{0}w_{0}\in \mathcal{M}$.
Then  from $t_{0}w_{n}\rightharpoonup t_{0}w_{0}$ and \eqref{eq1*} we get that
$$F(t_{0}w_{0})<F(w_{0})<\liminf\limits_{n\to \infty}F(w_{n})=m_{\lambda}$$
which is a contradiction. Hence $w_{n}\to w_{0}$
strongly in $H^{\frac{\alpha}{2}}_{0,L}(\mathcal{C}_{\Omega})$.
This implies $w_{0}\in \mathcal{M}$ and $F(w_{0})=m_{\lambda}$.
Moreover, it  follows from $F(w_{0})=F(|w_{0}|)$
and $|w_{0}|\in \mathcal{M}$ that $w_{0}$ is a
nonnegative weak solution to \eqref{eq20}.
Then by the strong maximum
principle \cite{ls2007} we have $w_{0}>0$ in $\mathcal{C}_{\Omega}$,
that is, $w_{0}$ is a positive solution of problem \eqref{eq20}.
\end{proof}

Now, we establish the existence of a  minimum for $J$ on $\mathcal{N}^{+}$.
\begin{proposition}\label{p3}
For  each $\lambda\in (0\,,\,\lambda_{*})$, the functional $J$
has a minimizer $w_{1}$ in $\mathcal{N}$.
\end{proposition}

\begin{proof}
From Lemma \ref{le6}, it is easily derived the coerciveness and
the lower boundedness of $J$ on $\mathcal{N}$.
Clearly, by the Ekeland's variational principle applying for the
minimization problem $\inf\limits_{\mathcal{N}}J(w)$,
there exists a minimizing sequence
$\{w_{n}\}\subset\mathcal{N}$ such that
\begin{equation}\label{eq29}
J(w_{n})<m_{J}+\frac{1}{n},
\end{equation}
and
\begin{equation}\label{eq30}
J(Z)\geq J(w_{n})-\frac{1}{n}\|w_{n}-Z\|_{H^{\frac{\alpha}{2}}_{0,L}(\mathcal{C}_{\Omega})},
\quad \forall \,Z\in \mathcal{N}.
\end{equation}

Let $w_{0}$ be a positive solution of \eqref{eq20}
 satisfying $F(w_{0})=m_{F}<0$.
Then
\begin{equation*}
 \aligned
m_{\lambda}&=F(w_{0})=\frac{\kappa_{\alpha}}{2}\int\limits_{\mathcal{C}_{\Omega}}y^{1-\alpha}|w_{0}|^{2}dxdy-
\frac{\lambda}{q}\int\limits_{\Omega\times\{0\}} f(x)|w_{0}|^{q}dx\\
 &=(\frac{1}{2}-\frac{1}{q})\|w_{0}\|^{2}_{H^{\frac{\alpha}{2}}_{0,L}(\mathcal{C}_{\Omega})},
  \endaligned
 \end{equation*}
that is,
\begin{equation}\label{eq25}
\|w_{0}\|^{2}_{H^{\frac{\alpha}{2}}_{0,L}(\mathcal{C}_{\Omega})}=\frac{2q}{q-2}m_{\lambda}>0.
\end{equation}
By Lemma 2.5 in \cite{wu2008}, for $w_{0}$,
there exists a positive constant $t_{1}$
such that $t_{1}w_{0}\in \mathcal{N}^{+}$, i.e.,
\begin{equation}\label{eq25*}
\int\limits_{\Omega\times\{0\}} |t_{1}w_{0}|^{2^*_{\alpha}}dx
<\frac{2-q}{2^*_{\alpha}-q}\,\kappa_{\alpha}
\int\limits_{\mathcal{C}_{\Omega}}y^{1-\alpha}|\nabla(t_{1}w_{0})|^{2}dxdy.
 \end{equation}
Then, from \eqref{eq25} and \eqref{eq25*},
\begin{equation*}
\aligned
J(t_{1}\,w_{0})
&=\frac{q-2}{2q}\,\kappa_{\alpha}\int\limits_{\mathcal{C}_{\Omega}}y^{1-\alpha}|\nabla(t_{1}w_{0})|^{2}dxdy
+\frac{2^*_{\alpha}-q}{q\,2^*_{\alpha}}\int\limits_{\Omega\times\{0\}} |t_{1}w_{0}|^{2^*_{\alpha}}dx\\
&<\Big{(}\frac{1}{2}-\frac{1}{2^*_{\alpha}}\Big{)}\frac{q-2}{q}\|t_{1}w_{0}\|^{2}
=\frac{\alpha}{2N}\frac{q-2}{q}t^{2}_{1}\,\frac{2q}{q-2}m_{\lambda}\\
&=\frac{\alpha}{N}t^{2}_{1}\,m_{\lambda}<0.
\endaligned\end{equation*}
This yields
\begin{equation}\label{eq26}
m_{J}\leq m^{+}<0.
\end{equation}\label{eq26*}
So \eqref{eq29}, \eqref{eq26} and the coerciveness of $J$
imply that the minimizer sequence $\{w_{n}\}$ is bounded,
and so there exists a subsequence of $\{w_{n}\}$, still denoted by $\{w_{n}\}$, and $w_{1}$ such that
\begin{equation*}\aligned
&w_{n}\rightharpoonup w_{1}\,\,\text{weakly in }\,\,H^{\frac{\alpha}{2}}_{0,L}(\mathcal{C}_{\Omega});\\
&w_{n}(\cdot,0)\to w_{1}(\cdot,0)\,\,\text{strongly in }\,\,L^{p}(\Omega)\,\text{for}\,\, 1\leq p< 2^*_{\alpha};\\
&w_{n}(\cdot,0)\to  w_{1}(\cdot,0)\,\,\text{a.e. in }\,\,\Omega.\\
\endaligned\end{equation*}
Now, we claim that $w_{1}\not\equiv0$.
In fact, suppose on the contrary that $w_{1}\equiv0$.
Since $w_{n}\in \mathcal{N}$, we deduce
\begin{equation*}\label{eq2*}
\aligned
J(w_{n})
&=\frac{\kappa_{\alpha}}{2}\int\limits_{\mathcal{C}_{\Omega}}y^{1-\alpha}|\nabla w_{n}|^{2}dxdy
-\frac{\lambda}{q}\int\limits_{\Omega\times\{0\}} f(x)|w_{n}|^{q}dx-\frac{1}{2^*_{\alpha}}\int\limits_{\Omega\times\{0\}} |w_{n}|^{2^*_{\alpha}}dx\\
&=\frac{2\alpha}{N}\|w_{n}\|^{2}_{H^{\frac{\alpha}{2}}_{0,L}(\mathcal{C}_{\Omega})}
-\lambda\,\frac{2^*_{\alpha}-q}{2^*_{\alpha}q}\int\limits_{\Omega\times\{0\}} f(x)|w_{n}|^{q}dx\\
&>-\lambda\,\frac{2^*_{\alpha}-q}{q2^*_{\alpha}}\int\limits_{\Omega\times\{0\}} f(x)|w_{n}|^{q}dx,\\
\endaligned\end{equation*}
that is,
\begin{equation*}\label{eq4*}
\int\limits_{\Omega\times\{0\}} f(x)|w_{n}|^{q}dx>
-\frac{q2^*_{\alpha}}{\lambda(2^*_{\alpha}-q)}J(w_{n})
\geq -\frac{q2^*_{\alpha}}{\lambda(2^*_{\alpha}-q)}m_{J}>0,
\end{equation*}
which clearly shows that $w_{1}\not\equiv0$.

Next, we will show that $$\|J'(w_{n})\|_{H^{\frac{\alpha}{2}}_{0,L}(\mathcal{C}_{\Omega})}\to 0\,\,\text{as}\,\,n\to \infty.$$
Exactly the same as in Lemma \ref{le7} we may apply suitable function
$t_{n}:\,B_{\varepsilon}(0)\to\mathbb{R}^{+}$ for some $\varepsilon>0$ small
such that
$$t_{n}(v)(w_{n}-v)\in \mathcal{N},\,\,\forall\,v\in H^{\frac{\alpha}{2}}_{0,L}(\mathcal{C}_{\Omega}),
\,\,\|v\|_{H^{\frac{\alpha}{2}}_{0,L}(\mathcal{C}_{\Omega})}<\varepsilon.$$

Set $\eta_{n}=t_{n}(v)(w_{n}-v)$. Since $\eta_{n}\in\mathcal{N}$,
 we deduce from \eqref{eq30} that
\begin{equation*}\label{eq34}
J(\eta_{n})-J(w_{n})\geq-\frac{1}{n}\|\eta_{n}-w_{n}\|_{H^{\frac{\alpha}{2}}_{0,L}(\mathcal{C}_{\Omega})}.
\end{equation*}
By the mean value theorem, we have
\begin{equation}\label{eq35}
\langle J'(w_{n})\,,\,\eta_{n}-w_{n}\rangle\geq -\frac{1}{n}\|\eta_{n}-w_{n}\|_{H^{\frac{\alpha}{2}}_{0,L}(\mathcal{C}_{\Omega})}
+o(\|\eta_{n}-w_{n}\|_{H^{\frac{\alpha}{2}}_{0,L}(\mathcal{C}_{\Omega})}).
\end{equation}
Thus, from $\eta_{n}-w_{n}=(t_{n}(v)-1)(w_{n}-v)-v$ and \eqref{eq35},
we get
\begin{equation}\label{eq36}\aligned
\langle J'(w_{n})\,,\,-v\rangle&+(t_{n}(v)-1)\langle J'(w_{n})\,,\,w_{n}-v\rangle\\
&\geq -\frac{1}{n}\|\eta_{n}-w_{n}\|_{H^{\frac{\alpha}{2}}_{0,L}(\mathcal{C}_{\Omega})}
+o(\|\eta_{n}-w_{n}\|_{H^{\frac{\alpha}{2}}_{0,L}(\mathcal{C}_{\Omega})}).\\
\endaligned\end{equation}

Let $v=\frac{r \,w_{1}}{\|w_{1}\|_{H^{\frac{\alpha}{2}}_{0,L}(\mathcal{C}_{\Omega})}}$,
$0<r<\varepsilon$. Substituting into \eqref{eq36}, we have
\begin{equation}\label{eq38}
\aligned
\langle J'(w_{n})\,,\,\frac{w_{1}}{\|w_{1}\|_{H^{\frac{\alpha}{2}}_{0,L}(\mathcal{C}_{\Omega})}}\rangle
&\leq \frac{1}{n\,r}\,\|\eta_{n}-w_{n}\|_{H^{\frac{\alpha}{2}}_{0,L}(\mathcal{C}_{\Omega})}
+\frac{1}{r}\,o(\|\eta_{n}-w_{n}\|_{H^{\frac{\alpha}{2}}_{0,L}(\mathcal{C}_{\Omega})})\\
&+\frac{(t_{n}(v)-1)}{r}\langle J'(w_{n})-J'(\eta_{n})\,,\,w_{n}-v\rangle.\\
\endaligned
\end{equation}
Since
\begin{equation}\label{eq38*}
\aligned
\|\eta_{n}-w_{n}\|_{H^{\frac{\alpha}{2}}_{0,L}(\mathcal{C}_{\Omega})}
&=\|(t_{n}(v)-1)\,w_{n}-t_{n}(v)\,v\|_{H^{\frac{\alpha}{2}}_{0,L}(\mathcal{C}_{\Omega})}\\
&\leq \varepsilon|t_{n}(v)|+|t_{n}(v)-1|\|w_{n}\|_{H^{\frac{\alpha}{2}}_{0,L}(\mathcal{C}_{\Omega})}
\endaligned\end{equation}
and
\begin{equation}\label{eq38**}
\lim\limits_{r\to 0}\frac{|t_{n}(v)-1|}{r}=\lim\limits_{r\to 0}\frac{|\langle t'_{n}(0)\,,\,v\rangle|}{r}\leq \|t'_{n}(0)\|_{H^{\frac{\alpha}{2}}_{0,L}(\mathcal{C}_{\Omega})}.
\end{equation}
If we let $r\to 0$ in the right hand of \eqref{eq38} for a fixed $n$,
then by \eqref{eq38*}, \eqref{eq38**} and the boundedness of $w_{n}$,
we can find a constant $C>0$ such that
\begin{equation}\label{eq39}
\langle J'(w_{n})\,,\,\frac{w_{1}}{\|w_{1}\|_{H^{\frac{\alpha}{2}}_{0,L}(\mathcal{C}_{\Omega})}}\rangle
\leq \frac{C}{n}\Big{(}1+\|t'_{n}(0)\|_{H^{\frac{\alpha}{2}}_{0,L}(\mathcal{C}_{\Omega})}\Big{)}.
\end{equation}

We are done once we show that $\|t'_{n}(0)\|_{H^{\frac{\alpha}{2}}_{0,L}(\mathcal{C}_{\Omega})}$
is uniformly bounded in $n$.
Since
$$\langle t'_{n}(0)\,,\,\varphi\rangle=\frac{2\kappa_{\alpha}\int\limits_{\mathcal{C}_{\Omega}}y^{1-\alpha} \nabla v_{n}\nabla \varphi dxdy
-q\lambda\int\limits_{\Omega\times\{0\}} f(x)|w_{n}|^{q-2}w_{n}\varphi dx-2^*_{\alpha}\int\limits_{\Omega\times\{0\}} |w|^{2^*_{\alpha}-2}w_{n}\varphi dx}
{(2-q) \| w_{n}|^{2}_{H^{\frac{\alpha}{2}}_{0,L}(\mathcal{C}_{\Omega})}
-(2^*_{\alpha}-q)\int\limits_{\Omega\times\{0\}} |w_{n}|^{2^*_{\alpha}}dx},$$
we have by the boundness of $w_{n}$,
\begin{equation}\label{eq40}
\| t'_{n}(0)\|_{H^{\frac{\alpha}{2}}_{0,L}(\mathcal{C}_{\Omega})}
\leq \frac{C_{1}}{|(2-q)\|w_{n}\|^{2}
-(2^*_{\alpha}-q)\int\limits_{\Omega\times\{0\}} |w_{n}|^{2^*_{\alpha}}dx|},
\end{equation}
for some  suitable positive constant $C_{1}$. We next only need to show that
\begin{equation}\label{eq40*}
\Big{|}(2-q)\|w_{n}\|^{2}_{H^{\frac{\alpha}{2}}_{0,L}(\mathcal{C}_{\Omega})}
-(2^*_{\alpha}-q)\int\limits_{\Omega\times\{0\}} |w_{n}|^{2^*_{\alpha}}dx\Big{|}\geq c>0
\end{equation}
for some $c>0$ and $n$ large enough.
Arguing by contradiction, assume that there exists a subsequence
 $\{w_{n}\}$ such that
 \begin{equation}\label{eq42}
 (2-q)\|w_{n}\|^{2}_{H^{\frac{\alpha}{2}}_{0,L}(\mathcal{C}_{\Omega})}
 -(2^*_{\alpha}-q)\int\limits_{\Omega\times\{0\}}\ |w_{n}|^{2^*_{\alpha}}dx\to 0\,\,\text{as}\,\,n\to \infty.
 \end{equation}
 Then,
\begin{equation}\label{eq43}
\aligned
\lim\limits_{n\to \infty}\int\limits_{\Omega\times\{0\}} |w_{n}|^{2^*_{\alpha}}dx
&=\lim\limits_{n\to \infty}\frac{2-q}{2^*_{\alpha}-q}\|w_{n}\|^{2}_{H^{\frac{\alpha}{2}}_{0,L}(\mathcal{C}_{\Omega})}\\
&\geq \frac{2-q}{2^*_{\alpha}-q}\|w_{1}\|^{2}_{H^{\frac{\alpha}{2}}_{0,L}(\mathcal{C}_{\Omega})}>0.
\endaligned\end{equation}
Therefore, we can find a constant $C_{2}>0$ such that
\begin{equation}\label{eq43}
\int\limits_{\Omega\times\{0\}} |w_{n}|^{2^*_{\alpha}}dx>C_{2}
\end{equation}
for $n$ large enough. In addition, \eqref{eq42} and the fact
that $w_{n}\in \mathcal{N}^{+}$ also give as
\begin{equation*}\label{eq44}
\lambda\int\limits_{\Omega\times\{0\}} f(x)|w_{n}|^{q}dx=\|u_{n}\|^{2}_{H^{\frac{\alpha}{2}}_{0,L}(\mathcal{C}_{\Omega})}
-\int\limits_{\Omega\times\{0\}} |w_{n}|^{2^*_{\alpha}}dx
=\frac{2\alpha}{(N-\alpha)(2-q)}\int\limits_{\Omega\times\{0\}} |w_{n}|^{2^*_{\alpha}}dx+o(1)
\end{equation*}
and
\begin{equation}\label{eq45}
\|w_{n}\|_{H^{\frac{\alpha}{2}}_{0,L}(\mathcal{C}_{\Omega})}\leq \Big{[}\lambda\frac{(2^*_{\alpha}-q)(N-\alpha)}{2\,\alpha}|f|_{L^{\infty}}\,S^{-\frac{N}{N-\alpha}}\Big{]}^{\frac{1}{2-q}}+o(1).
\end{equation}
This implies $K(w_{n})=o(1)$, where $K$ is given in Section 2.

However, by \eqref{eq43}, \eqref{eq45},
similar to the calculation of \eqref{eq00},
for each  $\lambda\in(0\,,\,\lambda_{*})$,
there is a $C_{3}>0$ such that
\begin{equation*}
K(w_{n})>C_{3},
\end{equation*}
which is impossible.

Hence, from \eqref{eq39}, \eqref{eq40} and \eqref{eq40*},
$$\langle J'(w_{n})\,,\,\frac{w_{1}}{\|w_{1}\|_{H^{\frac{\alpha}{2}}_{0,L}(\mathcal{C}_{\Omega})}}\rangle \leq \frac{C}{n}$$
for some  $C>0$. Taking $n\to \infty$, we get $\|J'(w_{n})\|_{H^{\frac{\alpha}{2}}_{0,L}(\mathcal{C}_{\Omega})}\to 0$.
This shows that $\{w_{n}\}$ is a (PS) sequence of functional $J$.

Finally, we prove that $w_{n}\to w_{1}$
strongly in $H^{\frac{\alpha}{2}}_{0,L}(\mathcal{C}_{\Omega})$.
 Since $w_{n}\rightharpoonup w_{1}$
weakly in $H^{\frac{\alpha}{2}}_{0,L}(\mathcal{C}_{\Omega})$,
it follows that
\begin{equation*}\aligned
m_{J}\leq J(w_{1})
&=\frac{1}{2}\|w_{1}\|^{2}_{H^{\frac{\alpha}{2}}_{0,L}(\mathcal{C}_{\Omega})}
-\frac{\lambda}{q}\int\limits_{\Omega\times\{0\}} f(x)|w_{1}|^{q}dx
-\frac{1}{2^*_{\alpha}}\int\limits_{\Omega\times\{0\}} |w_{1}|^{2^*_{\alpha}}dx\\
&=\frac{\alpha}{2N}\|w_{1}\|^{2}_{H^{\frac{\alpha}{2}}_{0,L}(\mathcal{C}_{\Omega})}
-\lambda\,\frac{2^*_{\alpha}-q}{q2^*_{\alpha}}\int\limits_{\Omega\times\{0\}} f(x)|w_{1}|^{q}dx\\
&\leq \lim\limits_{n\to \infty} J(w_{n})=m_{J}.
\endaligned
\end{equation*}
Consequently, $w_{n}\to w_{1}$ strongly in
$H^{\frac{\alpha}{2}}_{0}(\Omega)$ and $J(w_{1})=m_{J}$.
The proof is completed.
\end{proof}

\begin{theorem}\label{th01-1}
For each  $\lambda\in (0\,,\,\lambda_{*})$,
the problem \eqref{eq01}
admits a positive solution in $\mathcal{N}^{+}$.
\end{theorem}
\begin{proof}
From Proposition \ref{p3}, we have that $w_{1}$ is a nontrivial
solution of problem \eqref{eq02}.
Moreover, we have $$w_{1}\in \mathcal{N}^{+}.$$
In fact, if $w_{1}\in \mathcal{N}^{-}$, by Lemma \ref{le0},
there exists a unique $t^{-}(w_{1})>0$, $t^{+}(w_{1})>0$
such that $t^{-}(w_{1})\,w_{1}\in \mathcal{N}^{-}$, then we have $t^{-}(w_{1})=1$ and
$t^{+}(w_{1})<1$. Since
$J(t^{+}(w_{1})\,w_{1})=\min\limits_{t\in[0,t^{-}(w_{1})]}J(t\,w_{1})$,
we can find a $t_{0}\in (t^{+}(w_{1})\,,\,t^{-}(w_{1}))$ such that
$$J(t^{+}(w_{1})\,w_{1})<J(t_{0}\,w_{1})\leq J(t^{-}(w_{1})\,w_{1})=J(1\cdot w_{1})=m_{J},$$
which implies that $w_{1}\in \mathcal{N}^{+}$.
Since $J(w_{1})=J(|w_{1}|)$ and $|w_{1}|\in \mathcal{N}^{+}$, we can take $w_{1}\geq 0$.
By the strong maximum principle \cite{ls2007}, we get $w_{1}>0$ in $H^{\frac{\alpha}{2}}_{0,L}(\mathcal{C}_{\Omega})$.
Hence, $u_{1}(x)=w_{1}(x,0)\in H^{\frac{\alpha}{2}}_{0}(\Omega)$ is a positive solution of
problem \eqref{eq01} and $J(w_{1})=m^{+}$. We complete the proof.
\end{proof}

\begin{remark}\label{r2}
For $w_{1}\in \mathcal{N}^{+}$, by the H\"{o}lder inequality and the Young inequality we have
\begin{equation*}\aligned
0>J(w_{1})
&=\frac{\alpha}{2N}\int\limits_{\mathcal{C}_{\Omega}}y^{1-\alpha}|\nabla w_{1}|^{2}dxdy
-\lambda\,\frac{2^*_{\alpha}-q}{q\,2^*_{\alpha}}\int\limits_{\Omega\times\{0\}} f(x)|w|^{q}dx\\
&\geq \frac{\alpha}{2N}\|w_{1}\|^{2}_{H^{\frac{\alpha}{2}}_{0,L}(\mathcal{C}_{\Omega})}
-\lambda\,\frac{2^*_{\alpha}-q}{q\,2^*_{\alpha}}\,
|f|_{L^{\infty}}\,\Big(\kappa_{\alpha}S(\alpha,N)\Big)^{-\frac{q}{2}}\|w_{1}\|_{H^{\frac{\alpha}{2}}_{0,L}(\mathcal{C}_{\Omega})}^{q}\\
&\geq-\lambda\,\frac{2-q}{q\,2^*_{\alpha}}\Big(
|f|_{L^{\infty}}\,\kappa_{\alpha}S(\alpha,N)^{-\frac{q}{2}}\Big)^{\frac{2}{2-q}}.\\
\endaligned
\end{equation*}
So, we deduce that $J(w_{1})\to 0$ as $\lambda\to 0$.
\end{remark}

\vspace{5mm}

\subsection{The minimizer solution on $\mathcal{N}^{-}$}
In the following, we prove that problem \eqref{eq02}
has a solution in $\mathcal{N}^{-}$. Since $J$
is coercive and  bounded from below on $\mathcal{N}$
and so on $\mathcal{N}^{-}$, $$m^{-}=\inf\limits_{\mathcal{N}^{-}}J(w).$$
Then, there exists a minimizing sequence $\{w_{n}\}\subset \mathcal{N}^{-}$
 such that
\begin{equation}\label{eq90}
J(w_{n})\to m^{-}\quad \text{as}\,\,n\to \infty.
\end{equation}

\begin{lemma}\label{le8}
The set $\mathcal{N}^{-}$ is closed.
\end{lemma}
\begin{proof}
Suppose that there are some $w_{n}\in \mathcal{N}^{-}$ and $w_{n}\to w_{0}\not\in \mathcal{N}^{-}$,
 then $w_{0}\in \mathcal{N}^{0}=\{0\}$. For $w_{n}\in \mathcal{N}^{-}$, we have
$$0\leq (2-q)\kappa_{\alpha}\int\limits_{\mathcal{C}_{\Omega}}y^{1-\alpha}|\nabla w_{n}|^{2}dxdy
<(2^*_{\alpha}-q)\int\limits_{\Omega\times\{0\}}|w_{n}|^{2^*_{\alpha}}dx\to 0.$$
This implies that $\lim\limits_{n\to \infty}\kappa_{\alpha}\int\limits_{\mathcal{C}_{\Omega}}y^{1-\alpha}|\nabla w_{n}|^{2}dxdy=0$. Note that if $w_{n}\in \mathcal{N}^{-}$, then $\|w_{n}\|_{H^{\frac{\alpha}{2}}_{0,L}(\mathcal{C}_{\Omega})}\\ \geq \gamma>0$
for a suitable $\gamma>0$. This is a contradiction.
Hence we have $w_{0}\in \mathcal{N}^{-}$, and so $\mathcal{N}^{-}$ is closed.
\end{proof}

Next, we will use the trace inequality \eqref{eq04} to the family of minimizers
$w_{\varepsilon}=E_{\alpha}(u_{\varepsilon})$,
where $u_{\varepsilon}$ is given in \eqref{eq05}.

 Note that $f$ is a indefinite continuous function on
$\Omega$ and $f^{+}\not\equiv0$, where $f^{+}=\max\{f(x)\,,\,0\}$,
then the set $\Sigma:=\{x\in \Omega:\,f(x)>0\}\subset \Omega$ is an open set with positive measure.
Without loss of generality, we may assume that $\Sigma$ is a domain.

Let $\eta\in C^{\infty}_{0}(\mathcal{C}_{\Sigma})$,
$0\leq\eta\leq 1$( $\forall (x,y)\in \Sigma \times(0,\infty)$),
 be a positive function satisfying
$$\Big(\text{supp}f^{+}\times\{y>0\}\Big)
\cap \{(x,y)\in \mathcal{C}_{\Sigma}:\,\eta=1\}\not\neq \emptyset.$$
Moreover, for small fixed
$\rho>0$,
$$\eta(x,y)=\left\{
\aligned
&1,\quad &(x,y)\in B_{\rho},\\
&0,&(x,y)\not\in \overline{B_{2\rho}},\\
\endaligned
\right.$$
where $B_{\rho}=\{(x,y):\,|(x,y)|<\rho,\,y>0,\,x\in \Sigma\}$. We take $\rho$ small
enough such that $\overline{B_{2\rho}}\subset \overline{\mathcal{C}_{\Sigma}}$.
Note that $\eta\,w_{\varepsilon}\in H^{\frac{\alpha}{2}}_{0,L}(\mathcal{C}_{\Omega})$.

Let $\lambda_{*}>0$ be as in  \eqref{lambda}.
Then for $\lambda\in (0\,,\,\lambda_{*})$ we have the following result.
\begin{lemma}\label{l6}
Let $w_{1}$ be the local minimum in Proposition \ref{p3}. Then for $\varepsilon>0$ small enough,
\begin{equation*}\label{eq51}
\sup\limits_{t\geq0}J(w_{1}+t\eta w_{\varepsilon})<m_{J}+\frac{\alpha}{2N}\,\Big(\kappa_{\alpha}S(\alpha,N)\Big)^{\frac{N}{\alpha}}.
\end{equation*}

\end{lemma}

\begin{proof}
First, we have
\begin{equation}\label{eq52}
\aligned
J(w_{1}+t\eta w_{\varepsilon})
&=\frac{\kappa_{\alpha}}{2}\int\limits_{\mathcal{C}_{\Omega}}y^{1-\alpha}|\nabla(w_{1}+t \eta w_{\varepsilon})|^{2}dxdy
-\frac{\lambda}{q}\int\limits_{\Omega\times\{0\}} f(x)|w_{1}+t \eta w_{\varepsilon}|^{q}dx\\
&-\frac{1}{2^*_{\alpha}}\int\limits_{\Omega\times\{0\}}|w_{1}+t \eta w_{\varepsilon}|^{2^*_{\alpha}}dx\\
&=\frac{1}{2}\|w_{1}\|^{2}_{H^{\frac{\alpha}{2}}_{0,L}(\mathcal{C}_{\Omega})}+\frac{t^{2}}{2}\|\eta w_{\varepsilon}\|^{2}_{H^{\frac{\alpha}{2}}_{0,L}(\mathcal{C}_{\Omega})}
+t\langle w_{1}\,,\,\eta w_{\varepsilon}\rangle\\
&-\frac{\lambda}{q}\int\limits_{\Omega\times\{0\}} f(x)|w_{1}+t \eta w_{\varepsilon}|^{q}dx
-\frac{1}{2^*_{\alpha}}\int\limits_{\Omega\times\{0\}}|w_{1}+t \eta w_{\varepsilon}|^{2^*_{\alpha}}dx.\\
\endaligned
\end{equation}

Since $w_{1}$ is a solution of Eq.\eqref{eq02},
we get that
\begin{equation}\label{eq53**}
\frac{1}{2}\|w_{1}\|^{2}_{H^{\frac{\alpha}{2}}_{0,L}(\mathcal{C}_{\Omega})}
=J(w_{1})+\frac{\lambda}{q}\int\limits_{\Omega\times\{0\}} f(x)|w_{1}|^{q}dx
+\frac{1}{2^*_{\alpha}}\int\limits_{\Omega\times\{0\}}|w_{1}|^{2^*_{\alpha}}dx;
\end{equation}

\begin{equation}\label{eq53***}
t\langle w_{1}\,,\,\eta w_{\varepsilon}\rangle
=t\lambda\int\limits_{\Omega\times\{0\}} f(x)|w_{1}|^{q-1}\eta w_{\varepsilon}dx
+t\int\limits_{\Omega\times\{0\}}|w_{1}|^{2^*_{\alpha}-1}\,\eta w_{\varepsilon}dx;
\end{equation}

\begin{equation}\label{eq53}
\aligned
\int\limits_{\Omega\times\{0\}} |w_{1}+t \eta w_{\varepsilon}|^{2^*_{\alpha}}dx
&=\int\limits_{\Omega\times\{0\}}  |w_{1}|^{2^*_{\alpha}}dx
+t^{2^*_{\alpha}}\int\limits_{\Omega\times\{0\}}|\eta w_{\varepsilon}|^{2^*_{\alpha}}dx
+2^*_{\alpha}t\int\limits_{\Omega\times\{0\}} |w_{1}|^{2^*_{\alpha}-2}w_{1}\eta w_{\varepsilon}dx\\
&+2^*_{\alpha}t^{2^*_{\alpha}-1}\int\limits_{\Omega\times\{0\}}|\eta w_{\varepsilon}|^{2^*_{\alpha}-2}\eta w_{\varepsilon}w_{1}dx
+o(\varepsilon^{\frac{N-\alpha}{2}});
\endaligned
\end{equation}
and
\begin{equation}\label{eq53*}
\aligned
&\int\limits_{\Sigma\times\{0\}} f(x)\Big(|w_{1}+t\eta w_{\varepsilon}|^{q}-|w_{1}|^{q}
+qt|w_{1}|^{q-1}\eta w_{\varepsilon}\Big)dx\\
&=q\int\limits_{\Sigma\times\{0\}} f^{+}(x)\Big{\{}\int_{0}^{t\eta w_{\varepsilon}}(|w_{1}+\tau|^{q-1}+|w_{1}|^{q-1}\tau)d\tau\Big{\}}dx\\
&\geq q\int\limits_{\Sigma\times\{0\}} f^{+}(x)\Big{\{}\int_{0}^{t\eta w_{\varepsilon}}(|w_{1}+\tau|^{q-1}+|w_{1}|^{q-1}\tau)d\tau\Big{\}}dx\\
&\geq 0.\\
\endaligned
\end{equation}

Substituting \eqref{eq53**},\eqref{eq53***},\eqref{eq53} and \eqref{eq53*} in \eqref{eq52}
and using the fact that $\eta\in C^{\infty}_{0}(\mathcal{C}_{\Sigma})$, we obtain

\begin{equation*}\label{eq56}
\aligned
&J(w_{1}+t \eta w_{\varepsilon})\\
&=J(w_{1})
-\frac{\lambda}{q}\int\limits_{\Omega\times\{0\}} f(x)(|w_{1}+t\eta w_{\varepsilon}|^{q}-|w_{1}|^{q})dx
+t\langle w_{1}\,,\,\eta w_{\varepsilon}\rangle
-t\int\limits_{\Omega\times\{0\}} |w_{1}|^{2^*_{\alpha}-1}\eta w_{\varepsilon}dx\\
&+\frac{t^{2}}{2}\,\|\eta w_{\varepsilon}\|^{2}_{H^{\frac{\alpha}{2}}_{0,L}(\mathcal{C}_{\Omega})}
-\frac{t^{2^*_{\alpha}}}{2^*_{\alpha}}\int\limits_{\Omega\times\{0\}} |\eta w_{\varepsilon}|^{2^*_{\alpha}}dx
-t^{2^*_{\alpha}-1}\int\limits_{\Omega\times\{0\}} |\eta w_{\varepsilon}|^{2^*_{\alpha}-1}w_{1}dx
+o(\epsilon^{\frac{N-\alpha}{2}})\\
&=J(w_{1})
-\frac{\lambda}{q}\int\limits_{\Sigma\times\{0\}} f(x)(|w_{1}+t\eta w_{\varepsilon}|^{q}-|w_{1}|^{q}
+qt|w_{1}|^{q-1}\eta w_{\varepsilon})dx\\
&+\frac{t^{2}}{2}\,\|\eta w_{\varepsilon}\|^{2}_{H^{\frac{\alpha}{2}}_{0,L}(\mathcal{C}_{\Omega})}
-\frac{t^{2^*_{\alpha}}}{2^*_{\alpha}}\int\limits_{\Omega\times\{0\}} |\eta w_{\varepsilon}|^{2^*_{\alpha}}dx
-t^{2^*_{\alpha}-1}\int\limits_{\Omega\times\{0\}} |\eta w_{\varepsilon}|^{2^*_{\alpha}-1}w_{1}dx
+o(\epsilon^{\frac{N-\alpha}{2}})\\
&\leq J(w_{1})
+\frac{t^{2}}{2}\,\|\eta w_{\varepsilon}\|^{2}_{H^{\frac{\alpha}{2}}_{0,L}(\mathcal{C}_{\Omega})}
-\frac{t^{2^*_{\alpha}}}{2^*_{\alpha}}\int\limits_{\Omega\times\{0\}} |\eta w_{\varepsilon}|^{2^*_{\alpha}}dx\\
&-t^{2^*_{\alpha}-1}\int\limits_{\Omega\times\{0\}} |\eta w_{\varepsilon}|^{2^*_{\alpha}-1}w_{1}dx
+o(\epsilon^{\frac{N-\alpha}{2}}).\\
\endaligned
\end{equation*}
Since
\begin{equation*}
\aligned
\int\limits_{\Omega\times\{0\}}|\eta w_{\varepsilon}|^{2^*_{\alpha}-1}dx
&=\int\limits_{\Omega\times\{0\}}\Big[\frac{\eta \varepsilon^{\frac{N-\alpha}{2}}}{(\varepsilon^{2}+|x|^{2})^{\frac{N-\alpha}{2}}}\Big]^{\frac{N+\alpha}{N-\alpha}}dx\\
&=\int\limits_{\mathbb{R}^{N}}\frac{\varepsilon^{\frac{N+\alpha}{2}}}
{\varepsilon^{N+\alpha}(1+|z|^{2})^{\frac{N+\alpha}{2}}}\,\,\varepsilon^{N}dz\\
&=C\,\varepsilon^{\frac{N-\alpha}{2}}\int^{+\infty}_{0}\frac{1}{(1+r^{2})^{\frac{N+\alpha}{2}}}\\
&\leq C\,\varepsilon^{\frac{N-\alpha}{2}},
\endaligned
\end{equation*}
and from \cite{bc2012,szy2014}, we have
$$\|\eta w_{\varepsilon}\|^{2}_{H^{\frac{\alpha}{2}}_{0,L}(\mathcal{C}_{\Omega})}=\| w_{\varepsilon}\|^{2}_{H^{\frac{\alpha}{2}}_{0,L}(\mathcal{C}_{\Omega})}+O(\varepsilon^{N-\alpha}),$$
$$\int\limits_{\Omega\times\{0\}} |\eta w_{\varepsilon}|^{2^*_{\alpha}}dx
=\int\limits_{\mathbb{R}^{N}}\Big(\frac{\varepsilon}{\varepsilon^{2}+|x|^{2}}\Big)^{N}dx+O(\varepsilon^N).$$
Thus,
\begin{equation}\label{eq56}
J(w_{1}+t \eta w_{\varepsilon})
\leq J(w_{1})+
\frac{t^{2}}{2}\,\|w_{\varepsilon}\|^{2}_{H^{\frac{\alpha}{2}}_{0,L}(\mathcal{C}_{\Omega})}
-\frac{t^{2^*_{\alpha}}}{2^*_{\alpha}}\int\limits_{\Omega\times\{0\}} | w_{\varepsilon}|^{2^*_{\alpha}}dx+O(\varepsilon^N)-C\,\varepsilon^{\frac{N-\alpha}{2}}+o(\epsilon^{\frac{N-\alpha}{2}})
\end{equation}

Let
$$h(t)=\frac{t^{2}}{2}\,\| w_{\varepsilon}\|^{2}_{H^{\frac{\alpha}{2}}_{0,L}(\mathcal{C}_{\Omega})}
-\frac{t^{2^*_{\alpha}}}{2^*_{\alpha}}\int\limits_{\Omega\times\{0\}} | w_{\varepsilon}|^{2^*_{\alpha}}dx.
$$
for $t>0$. Since $h(t)$ goes to $-\infty$ as $t\to \infty$, $\sup\limits_{t\geq 0}h(t)$ is
achieved at some $t_{\varepsilon}>0$ with $h'(t_{\varepsilon})=0$. That is
$$
0=\| w_{\varepsilon}\|^{2}_{H^{\frac{\alpha}{2}}_{0,L}(\mathcal{C}_{\Omega})}
-t_{\varepsilon}^{2^*_{\alpha}-2}\int\limits_{\Omega\times\{0\}} | w_{\varepsilon}|^{2^*_{\alpha}}dx
.$$
Therefore,
\begin{equation}\label{eq57}
h(t)\leq h(t_{\varepsilon})
=\Big(\frac{1}{2}-\frac{1}{2^*_\alpha}\Big)
\|w_{\varepsilon}\|^{\frac{2^*_{\alpha}}{2^*_{\alpha}-2}}_{H^{\frac{\alpha}{2}}_{0,L}(\mathcal{C}_{\Omega})}
\Big(\int\limits_{\Omega\times\{0\}} |w_{\varepsilon}|^{2^*_{\alpha}}dx\Big)^{-\frac{2}{2^*_{\alpha}-2}}.
\end{equation}

On the other hand, since $w_{\varepsilon}$ are minimizers of the trace inequality of \eqref{eq04}, we have that
\begin{equation}\label{eq91}
\|w_{\varepsilon}\|^{2}_{H^{\frac{\alpha}{2}}_{0,L}(\mathcal{C}_{\Omega})}=\kappa_{\alpha}\,S(\alpha,N)\,
\Big(\int\limits_{\Omega\times\{0\}} |w_{\varepsilon}|^{2^*_{\alpha}}dx\Big)^{\frac{2}{2^*_{\alpha}}}.
\end{equation}
Hence, from \eqref{eq56},\eqref{eq57} and \eqref{eq91}, we obtain
\begin{equation*}
\aligned
J(w_{1}+t\eta w_{\varepsilon})
&\leq J(w_{1})+\frac{\alpha}{2N}\Big(\kappa_{\alpha}\,S(\alpha,N)\Big)^{\frac{N}{\alpha}}+O(\varepsilon^N)
-C\varepsilon^{\frac{N-\alpha}{2}}+o(\varepsilon^{\frac{N-4}{2}})\\
&<m_{J}+\frac{\alpha}{2N}\Big(\kappa_{\alpha}\,S(\alpha,N)\Big)^{\frac{N}{\alpha}},
\endaligned
\end{equation*}
for $\varepsilon>0$ sufficiently small.
\end{proof}

The following proposition provides a precise description of the (PS)-sequence of $J$.
\begin{proposition}\label{p2}
If every minimizing sequence $\{u_{n}\}$ of $J$ on $\mathcal{N}^{-}$ satisfies
$m_{J}\leq J(w_{n})<m_{J}+\frac{\alpha}{2N}(\kappa_{\alpha}S)^{\frac{N}{\alpha}}$,
then $\{w_{n}\}$ satisfies the {\em (PS)}-condition on  $\mathcal{N}^{-}$.
\end{proposition}

\begin{proof}
By \eqref{eq90} and $\{w_{n}\}\subset \mathcal{N}^{-}$, it is easy to prove that the sequence
 $\{w_{n}\}$ is bounded in $H^{\frac{\alpha}{2}}_{0,L}(\mathcal{C}_{\Omega})$.
 Them we can extract a subsequence, still denoted by $\{w_{n}\}$, and $w_{2}$ such that, as $n\to \infty$,
\begin{equation}\label{eq58}
\aligned
&w_{n}\rightharpoonup w_{2} \quad \text{weakly in}\,\,H^{\frac{\alpha}{2}}_{0,L}(\mathcal{C}_{\Omega});\\
&w_{n}(\cdot,0)\to w_{2}(\cdot,0)\quad\text{strongly in}\,\,L^{p}(\Omega), \forall 1\leq p< 2^{*}_{\alpha};\\
&w_{n}(\cdot,0)\to w_{2}(\cdot,0)\quad\text{a.e. in}\,\,\Omega.
\endaligned\end{equation}

Since $\{w_{n}\}\subset \mathcal{N}^{-}$ is a minimizing sequence, by the Lagrange multiplier method, we get that $J'(w_{n})\to 0$ as $n\to \infty$. Consequently, by \eqref{eq58} we have
$$\langle J'(w_{2})\,,\,\varphi\rangle=0,\quad \forall \varphi\in H^{\frac{\alpha}{2}}_{0,L}(\mathcal{C}_{\Omega}).$$
Then $w_{2}$ is a solution in $H^{\frac{\alpha}{2}}_{0,L}(\mathcal{C}_{\Omega})$ for problem \eqref{eq02},
 and $J(w_{2})\geq m_{J}$.

First, we claim that $w_{2}\not\equiv0$. If not, by \eqref{eq58} we have
$\int\limits_{\Omega\times\{0\}} f(x)|w_{2}|^{q}dx\to 0$ as $n\to \infty$.
Thus, form $J'(w_{n})\to 0$, we obtain that
\begin{equation}\label{eq59}
\kappa_{\alpha}\int\limits_{\mathcal{C}_{\Omega}}y^{1-\alpha}|\nabla w_{n}|^{2}dxdy=\int\limits_{\Omega\times\{0\}} |w_{n}|^{2^*_{\alpha}}dx+o(1).
\end{equation}
and
\begin{equation*}\label{eq60}
\aligned
J(w_{n})
&=\frac{\kappa_{\alpha}}{2}\int\limits_{\mathcal{C}_{\Omega}}y^{1-\alpha}|\nabla w_{n}|^{2}dxdy-\frac{\lambda}{q}\int\limits_{\Omega\times\{0\}}f(x)|w_{n}|^{q}dx
-\frac{1}{2^*_{\alpha}}\int\limits_{\Omega\times\{0\}}|w_{n}|^{2^*_{\alpha}}dx\\
&=\frac{\alpha }{2N}\int\limits_{\Omega\times\{0\}}|w_{n}|^{2^*_{\alpha}}dx
<m_{J}+\frac{\alpha}{2N}(\kappa_{\alpha}\,S(\alpha,N))^{\frac{N}{\alpha}}\\
&<\frac{\alpha}{2N}(\kappa_{\alpha}\,S(\alpha,N))^{\frac{N}{\alpha}}\,(\text{since}\,\,m_{J}<0).\\
\endaligned
\end{equation*}
So, we get
\begin{equation}\label{eq60*}
\int\limits_{\Omega\times\{0\}}|w_{n}|^{2^*_{\alpha}}dx<(\kappa_{\alpha}S(\alpha,N))^{\frac{N}{\alpha}}.
\end{equation}
On the other hand, from \eqref{eq59} and \eqref{eq04}, we have that $$\int\limits_{\Omega\times\{0\}}|w_{n}|^{2^*_{\alpha}}dx\geq (\kappa_{\alpha}S(\alpha,N))^{\frac{N}{\alpha}}.$$
This contradicts \eqref{eq60*}.
Then $w_{2}\not\equiv0$ and $J(w_{2})\geq m_{J}$.

We write $\hat{w}_{n}=w_{n}-w_{2}$ with $\hat{w}_{n}\rightharpoonup 0$
weakly in $H^{\frac{\alpha}{2}}_{0,L}(\mathcal{C}_\Omega)$. By the Brezis-Lieb Lemma, we have
\begin{equation*}\label{eq61}
\int\limits_{\Omega\times\{0\}} |\hat{w}_{n}|^{2^*_{\alpha}}dx
=\int\limits_{\Omega\times\{0\}} |w_{n}-w_{2}|^{2^*_{\alpha}}dx
=\int\limits_{\Omega\times\{0\}} |w_{n}|^{2^*_{\alpha}}dx
-\int\limits_{\Omega\times\{0\}} |w_{2}|^{2^*_{\alpha}}dx+o(1).
\end{equation*}
Hence, for $n$ large enough, we can conclude that
\begin{equation*}\label{eq62}
\aligned
m_{J}+\frac{\alpha}{2N}(\kappa_{\alpha}S(\alpha,N))^{\frac{N}{\alpha}}
&>J(w_{2}+\hat{w}_{n})\\
&=J(w_{2})+\frac{\kappa_{\alpha}}{2}\int\limits_{\mathcal{C}_{\Omega}}y^{1-\alpha}|\nabla \hat{w}_{n}|^{2}dxdy-\frac{1}{2^*_{\alpha}}\int\limits_{\Omega\times\{0\}} |\hat{w}_{n}|^{2^*_{\alpha}}dx+o(1)\\
&\geq m_{J}+\frac{\kappa_{\alpha}}{2}\int\limits_{\mathcal{C}_{\Omega}}y^{1-\alpha}|\nabla \hat{w}_{n}|^{2}dxdy-\frac{1}{2^*_{\alpha}}\int\limits_{\Omega\times\{0\}} |\hat{w}_{n}|^{2^*_{\alpha}}dx+o(1),\\
\endaligned
\end{equation*}
this is,
\begin{equation}\label{eq63}
\frac{\kappa_{\alpha}}{2}\int\limits_{\mathcal{C}_{\Omega}}y^{1-\alpha}|\nabla \hat{w}_{n}|^{2}dxdy-\frac{1}{2^*_{\alpha}}\int\limits_{\Omega\times\{0\}} |\hat{w}_{n}|^{2^*_{\alpha}}dx<\frac{\alpha}{2N}(\kappa_{\alpha}S(\alpha,N))^{\frac{N}{\alpha}}+o(1).\\
\end{equation}

Since $J'(w_{n})\to 0$ as $n\to \infty$, $\{w_{n}\}$ is uniformly bounded and $w_{2}$ is a solution
of Eq. \eqref{eq02}, it  follows
\begin{equation*}\label{eq64}
\aligned
o(1)&=\langle J'(w_{n})\,,\,w_{n}\rangle\\
&=J'(w_{2})+\kappa_{\alpha}\int\limits_{\mathcal{C}_{\Omega}}y^{1-\alpha}|\nabla\hat{w}_{n}|^{2}dxdy
-\int\limits_{\Omega\times\{0\}} |\hat{w}_{n}|^{2^*_{\alpha}}dx+o(1)\\
&=\kappa_{\alpha}\int\limits_{\mathcal{C}_{\Omega}}y^{1-\alpha}|\nabla\hat{w}_{n}|^{2}dxdy
-\int\limits_{\Omega\times\{0\}} |\hat{w}_{n}|^{2^*_{\alpha}}dx+o(1).\\
\endaligned
\end{equation*}
We obtain
\begin{equation}\label{eq65}
\kappa_{\alpha}\int\limits_{\mathcal{C}_{\Omega}}y^{1-\alpha}|\nabla\hat{w}_{n}|^{2}dxdy
=\int\limits_{\Omega\times\{0\}} |\hat{w}_{n}|^{2^*_{\alpha}}dx+o(1)\quad (n\to \infty).
\end{equation}
We claim that \eqref{eq63} and \eqref{eq65} can hold simultaneously only if $\{\hat{w}_{n}\}$ admits a subsequence which converges strongly to zero.
 If not, then $\|\hat{w}_{n}\|_{H^{\frac{\alpha}{2}}_{0,L}(\mathcal{C}_{\Omega})}$ is bounded away from zero, that is $\|\hat{w}_{n}\|_{H^{\frac{\alpha}{2}}_{0,L}(\mathcal{C}_{\Omega})}>c>0$.
From \eqref{eq65} and \eqref{eq04} then it follows
\begin{equation}\label{eq65*}
\int\limits_{\Omega\times\{0\}} |\hat{w}_{n}|^{2^*_{\alpha}}dx
\geq \Big(\kappa_{\alpha}S(\alpha,N)\Big)^{\frac{N}{\alpha}}+o(1).
\end{equation}

By \eqref{eq63}, \eqref{eq65} and \eqref{eq65*}, for $n$ large enough, we have
\begin{equation*}\label{eq66}
\aligned
\frac{\alpha}{2N}\Big(\kappa_{\alpha}S(\alpha,N)\Big)^{\frac{N}{4}}
&\leq \frac{\alpha}{2N}\int\limits_{\Omega\times\{0\}} |\hat{w}_{n}|^{2^*_{\alpha}}dx+o(1)\\
&=\frac{\kappa_{\alpha}}{2}\int\limits_{\mathcal{C}_{\Omega}}y^{1-\alpha}|\nabla \hat{w}_{n}|^{2}dxdy-\frac{1}{2^*_{\alpha}}\int\limits_{\Omega\times\{0\}}|\hat{w}_{n}|^{2^*_{\alpha}}dx+o(1)\\
&<\frac{\alpha}{2N}\Big(\kappa_{\alpha}S(\alpha,N)\Big)^{\frac{N}{\alpha}},
\endaligned
\end{equation*}
which is a contradiction. Consequently, $w_{n}\to w_{2}$ strongly in $H^{\frac{\alpha}{2}}_{0,L}(\mathcal{C}_\Omega)$
and $w_{2}\in \mathcal{N}^{-}$.
\end{proof}

Next, we establish the existence of a local minimum for $J$ on $\mathcal{N}^{-}$.
\begin{proposition}\label{p4}
For  any $\lambda\in (0\,,\,\lambda_{*})$, the functional $J$
has a minimizer $w_{2}\in \mathcal{N}^{-}$ such that
$$J(w_{2})=m_{-}<m_{J}+\frac{\alpha}{2N}\Big(\kappa_{\alpha}S(\alpha,N)\Big)^{\frac{N}{\alpha}}.$$
\end{proposition}

\begin{proof}
For every $w\in H^{\frac{\alpha}{2}}_{0,L}(\mathcal{C}_{\Omega})$, by Lemma \ref{le0}, we can find a unique
$t^{-}(w)>0$ such that $t^{-}(w)w\in \mathcal{N}^{-}$. Define
\begin{equation*}\aligned
&W_{1}=\{w:\,w=0\,\,\text{or}\,\,t^{-}\Big(\frac{w}{\|w\|_{H^{\frac{\alpha}{2}}_{0,L}(\mathcal{C}_\Omega)}}\Big)
>\|w\|_{H^{\frac{\alpha}{2}}_{0,L}(\mathcal{C}_\Omega)}\},\\
&W_{2}=\{w:\,t^{-}\Big(\frac{w}{\|w\|_{H^{\frac{\alpha}{2}}_{0,L}(\mathcal{C}_\Omega)}}\Big)
<\|w\|_{H^{\frac{\alpha}{2}}_{0,L}(\mathcal{C}_\Omega)}\}.
\endaligned
\end{equation*}
Then $\mathcal{N}^{-}$ disconnects $H^{\frac{\alpha}{2}}_{0,L}(\mathcal{C}_\Omega)$ in two connected components $W_{1}$ and $W_{2}$ and
$H^{\frac{\alpha}{2}}_{0,L}(\mathcal{C}_\Omega)\setminus\mathcal{N}^{-}=W_{1}\cup W_{2}$.

For each $w\in \mathcal{N}^{+}$, there exist unique $t^{-}(\frac{w}{\|w\|_{H^{\frac{\alpha}{2}}_{0,L}(\mathcal{C}_\Omega)}})>0$ and $t^{+}(\frac{w}{\|w\|_{H^{\frac{\alpha}{2}}_{0,L}(\mathcal{C}_\Omega)}})>0$ such that
$$t^{+}\Big(\frac{w}{\|w\|_{H^{\frac{\alpha}{2}}_{0,L}(\mathcal{C}_\Omega)}}\Big)
<t_{max}<t^{-}\Big(\frac{w}{\|w\|_{H^{\frac{\alpha}{2}}_{0,L}(\mathcal{C}_\Omega)}}\Big);$$
$$ t^{+}\Big(\frac{w}{\|w\|_{H^{\frac{\alpha}{2}}_{0,L}(\mathcal{C}_\Omega)}}\Big)
\frac{w}{\|w\|_{H^{\frac{\alpha}{2}}_{0,L}(\mathcal{C}_\Omega)}}\in \mathcal{N}^{+};$$
and
$$t^{-}\Big(\frac{w}{\|w\|_{H^{\frac{\alpha}{2}}_{0,L}(\mathcal{C}_\Omega)}}\Big)
\frac{w}{\|w\|_{H^{\frac{\alpha}{2}}_{0,L}(\mathcal{C}_\Omega)}}\in \mathcal{N}^{-}.$$
Since $w\in \mathcal{N}^{+}$, we have $t^{+}(\frac{w}{\|w\|_{H^{\frac{\alpha}{2}}_{0,L}(\mathcal{C}_\Omega)}})
\frac{1}{\|w\|_{H^{\frac{\alpha}{2}}_{0,L}(\mathcal{C}_\Omega)}}=1$.
By the fact that
$t^{+}(\frac{w}{\|w\|_{H^{\frac{\alpha}{2}}_{0,L}(\mathcal{C}_\Omega)}})
<t^{-}(\frac{w}{\|w\|_{H^{\frac{\alpha}{2}}_{0,L}(\mathcal{C}_\Omega)}})$, we get
$$t^{-}\Big(\frac{w}{\|w\|_{H^{\frac{\alpha}{2}}_{0,L}(\mathcal{C}_\Omega)}}\Big)>
\|w\|_{H^{\frac{\alpha}{2}}_{0,L}(\mathcal{C}_\Omega)},$$
and then $\mathcal{N}^{+}\subset W_{1}$.
In particular, $w_{1}\in W_{1}$ is the minimizer of $J$ in $\mathcal{N}^{+}$.

Now, we claim that there exists $l_{0}>0$ such that $w_{1}+l_{0}\eta w_{\varepsilon}\in W_{2}$.
First, we find a constant $c>0$ such that $0<t^{-}(\frac{w_{1}+l \eta w_{\varepsilon}}{\|w_{1}+l\eta w_{\varepsilon}\|_{H^{\frac{\alpha}{2}}_{0,L}(\mathcal{C}_{\Omega})}})<c$
for each $l>0$. Otherwise, there exists a sequence $\{l_{n}\}$ such that $l_{n}\to \infty$ and
 $t^{-}(\frac{w_{1}+l_{n}\eta w_{\varepsilon}}{\|w_{1}+l_{n} \eta w_{\varepsilon}\|_{H^{\frac{\alpha}{2}}_{0,L}(\mathcal{C}_{\Omega})}})\to \infty$ as $n\to \infty$.
 Let $\widetilde{w}_{n}=\frac{w_{1}+l_{n}\eta w_{\varepsilon}}{\|w_{1}+l_{n}\eta w_{\varepsilon}\|_{H^{\frac{\alpha}{2}}_{0,L}(\mathcal{C}_{\Omega})}}$. By Lemma \ref{le0}, we obtain $t^{-}(\widetilde{w}_{n})\widetilde{w}_{n}\in \mathcal{N}^{-}$. Then we have
 \begin{equation*}\label{eq67}
 \aligned
\int\limits_{\Omega\times\{0\}} |\widetilde{w}_{n}|^{2^*_{\alpha}} dx
&=\frac{1}{\|w_{1}+l_{n}\eta w_{\varepsilon}\|^{2^*_{\alpha}}_{H^{\frac{\alpha}{2}}_{0,L}(\mathcal{C}_{\Omega})}}\,\,
\int\limits_{\Omega\times\{0\}} |w_{1}+l_{n}\eta w_{\varepsilon}|^{2^*_{\alpha}}dx\\
&=\frac{1}{\|\frac{w_{1}}{l_{n}}+\eta w_{\varepsilon}\|^{2^*_{\alpha}}_{H^{\frac{\alpha}{2}}_{0,L}(\mathcal{C}_{\Omega})}}\,\,\int\limits_{\Omega\times\{0\}} |\frac{w_{1}}{l_{n}}+\eta w_{\varepsilon}|^{2^*_{\alpha}}dx\\
 &\to \frac{1}{\|\eta w_{\varepsilon}\|^{2^*_{\alpha}}_{H^{\frac{\alpha}{2}}_{0,L}(\mathcal{C}_{\Omega})}}\,\,\,\int\limits_{\Omega\times\{0\}} |\eta w_{\varepsilon}|^{2^*_{\alpha}}dx>0 (n\to \infty),
 \endaligned
 \end{equation*}
and
 \begin{equation*}\label{eq68}
\aligned
&J(t^{-}(\widetilde{w}_{n})\widetilde{w}_{n})\\
&=\frac{1}{2}[t^{-}(\widetilde{w}_{n})]^{2}
-\frac{\lambda}{q}[t^{-}(\widetilde{w}_{n})]^{q}\int\limits_{\Omega\times\{0\}} f(x)\widetilde{w}^{q}_{n}dx
-\frac{[t^{-}(\widetilde{w}_{n})]^{2^*_{\alpha}}}{2^*_{\alpha}}\int\limits_{\Omega\times\{0\}} |\widetilde{w}_{n}|^{2^*_{\alpha}}dx\\
&\to -\infty\quad (n\to \infty).
\endaligned
\end{equation*}
 This contradicts that $J$ is bounded below on $\mathcal{N}$.

Let
$$l_{0}=\frac{\sqrt{\Big{|}c^{2}-\|w_{1}\|^{2}_{H^{\frac{\alpha}{2}}_{0,L}(\mathcal{C}_{\Omega})}\Big{|}}}
{\|\eta w_{\varepsilon}\|_{H^{\frac{\alpha}{2}}_{0,L}(\mathcal{C}_{\Omega})}}+1.$$
Then
\begin{equation*}\label{eq69}
\aligned
\|w_{1}+l_{0}\eta &w_{\varepsilon}\|^{2}_{H^{\frac{\alpha}{2}}_{0,L}(\mathcal{C}_{\Omega})}
=\|w_{1}\|^{2}_{H^{\frac{\alpha}{2}}_{0,L}(\mathcal{C}_{\Omega})}
+l^{2}_{0}\|\eta w_{\varepsilon}\|^{2}_{H^{\frac{\alpha}{2}}_{0,L}(\mathcal{C}_{\Omega})}
+2l_{0}\langle w_{1}\,,\,\eta w_{\varepsilon}\rangle\\
&\geq \|w_{1}\|^{2}_{H^{\frac{\alpha}{2}}_{0,L}(\mathcal{C}_{\Omega})}
+\Big|c^{2}-\|w_{1}\|^{2}_{H^{\frac{\alpha}{2}}_{0,L}(\mathcal{C}_{\Omega})}\Big|
+2l_{0}\langle w_{1}\,,\,\eta w_{\varepsilon}\rangle\\
&\geq c^{2}\\
&>[t^{-}\Big(\frac{w_{1}+l_{0}\eta w_{\varepsilon}}{\|w_{1}+l_{0}\eta w_{\varepsilon}\|_{H^{\frac{\alpha}{2}}_{0,L}(\mathcal{C}_{\Omega})}}\Big)]^{2},
\endaligned
\end{equation*}
that is, $w_{1}+l_{0}\eta w_{\varepsilon}\in W_{2}$. Now, we define
\begin{equation*}\label{eq70}
\beta=\inf\limits_{\gamma\in \Gamma}\max\limits_{s\in[0,1]}J(\gamma(s)),
\end{equation*}
where $\Gamma=\{\gamma\in C([0,1]\,,\,H^{\frac{\alpha}{2}}_{0,L}(\mathcal{C}_\Omega)):
\,\gamma(0)=w_{1}\quad \text{and}\,\,\gamma(1)=w_{1}+l_{0}\eta w_{\varepsilon}\}$.
Define a path $\gamma(s)=w_{1}+sl_{0}\eta w_{\varepsilon}$ for $s\in[0\,,\,1]$, and
we have $\gamma(0)\in W_{1}$, $\gamma(1)\in W_{2}$.
Then there exists $s_{0}\in (0\,,\,1)$ such that $\gamma(s_{0})\in\mathcal{N}^{-}$, and
we have $\beta>m_{-}$. Therefore, by Lemma \ref{l6}, we get
$$m_{-}\leq \beta<m_{J}+\frac{\alpha}{2N}\Big(\kappa_{\alpha}S(\alpha,N)\Big)^{\frac{N}{\alpha}}.$$

Analogously to the proof of Proposition \ref{p3}, one can show that  Ekeland's
variational principle gives a sequence $\{w_{n}\}\in \mathcal{N}^{-}$ which satisfies
$$J(w_{n})\to m_{-}\quad\text{and}\quad J'(w_{n})\to 0\quad \text{as}\,\,n\to \infty.$$
Since $m_{-}<m_{J}+\frac{\alpha}{2N}\Big(\kappa_{\alpha}S(\alpha,N)\Big)^{\frac{N}{\alpha}}$,
by Proposition \ref{p2} and Lemma \ref{le8}, there exists a subsequence
$\{w_{n}\}$ and $w_{2}$ such that
$$w_{n}\to w_{2}\quad\text{strongly in }\,\, H^{\frac{\alpha}{2}}_{0,L}(\mathcal{C}_\Omega),$$
$w_{2}\in \mathcal{N}^{-}$ and $J(w_{2})=m_{-}$.

Since $J(w_{2})=J(|w_{2}|)$
and $|w_{2}|\in \mathcal{N}^{-}$, we can always take $w_{2}\geq0$. By the maximum principle \cite{ls2007},
we get $w_{2}>0$ in $H^{\frac{\alpha}{2}}_{0,L}(\mathcal{C}_{\Omega})$.
Hence, $u_{2}(x)=w_{2}(\cdot,0)\in H^{\frac{\alpha}{2}}_{0}(\Omega)$ is
a positive solution of problem \eqref{eq01}. The proof is completed.
\end{proof}

{\bf Proof of Theorem \ref{th01}.}
 By Theorem \ref{th01-1} and Proposition \ref{p4}, the equation \eqref{eq02}
has two positive solutions $w_{1}$ and $w_{2}$ such that $w_{1}\in \mathcal{N}^{+}$
and $w_{2}\in \mathcal{N}^{+}$. Since $\mathcal{N}^{+}\cap\mathcal{N}^{-}=\emptyset$.
This implies that problem \eqref{eq01} has at least two positive solutions
 $u_{1}(x)=w_{1}(x,0)$ and $u_{2}(x)=w_{2}(x,0)$.

\section{Concentration Behavior}

In this section, we give the proof of Theorem \ref{th02}.

For every $\mu>0$, we define
$$J_{\mu}(w)=\frac{\kappa_{\alpha}}{2}\int\limits_{\mathcal{C}_{\Omega}}y^{1-\alpha}|\nabla w|^{2}dxdy
-\frac{\mu}{2^*_{\alpha}}\int\limits_{\Omega\times\{0\}}|w|^{2^*_{\alpha}}dx;$$
$$\mathcal{O}_{\mu}=\{w\in H^{\frac{\alpha}{2}}_{0,L}(\mathcal{C}_{\Omega}):w\not\equiv0\,\,\text{and}\,\,
\langle J'_{\mu}(w)\,,\,w\rangle=0\}.$$

We have the following lemmas.

\begin{lemma}\label{le9}
For every $w\in \mathcal{N}^{-}$, there is a unique $t(w)>0$ such that $t(w)w\in \mathcal{O}_{1}$ and
\begin{equation}\label{eq72*}
1-\lambda|f|_{L^{\infty}}\,\Big{(}\frac{2^*_{\alpha}-q}{S_{0}(2-q)}\Big{)}^{\frac{2^*_{\alpha}-q}{2^*_{\alpha}-2}}\leq t^{2^*_{\alpha}-2}(w)\leq 1+\lambda|f|_{L^{\infty}}\,\Big{(}\frac{2^*_{\alpha}-q}{S_{0}(2-q)}\Big{)}^{\frac{2^*_{\alpha}-q}{2^*_{\alpha}-2}},
\end{equation}
where $S_{0}=\kappa_{\alpha}\,S(\alpha,N)$.
\end{lemma}

\begin{proof}
For each $w\in \mathcal{N}^{-}$, we have
\begin{equation}\label{eq72}
\kappa_{\alpha}\int\limits_{\mathcal{C}_{\Omega}}y^{1-\alpha}|\nabla w|^{2}dxdy
-\lambda\int\limits_{\Omega\times\{0\}}f(x)|w|^{q}dx
-\int\limits_{\Omega\times\{0\}}|w|^{2^*_{\alpha}}dx=0
\end{equation}
and
\begin{equation}\label{eq84}
0<(2-q)\kappa_{\alpha}\int\limits_{\mathcal{C}_{\Omega}}y^{1-\alpha}|\nabla w|^{2}dxdy
<(2^*_{\alpha}-q)\int\limits_{\Omega\times\{0\}}|w|^{2^*_{\alpha}}dx.
\end{equation}
Thus, from \eqref{eq84}, the functional
$$J_{1}(tw)
=t^{2}\,\frac{\kappa_{\alpha}}{2}\int\limits_{\mathcal{C}_{\Omega}}y^{1-\alpha}|\nabla w|^{2}dxdy
-\frac{t^{2^*_{\alpha}}}{2^*_{\alpha}}\int\limits_{\Omega\times\{0\}}|w|^{2^*_{\alpha}}dx$$ with respect to $t$ is initially increasing and eventually
decreasing and with a single turning point $t(w)$ such that $t(w)w\in \mathcal{O}_{1}$.
So
\begin{equation}\label{eq84*}
t^{2}(w)\kappa_{\alpha}\int\limits_{\mathcal{C}_{\Omega}}y^{1-\alpha}|\nabla w|^{2}dxdy
=t^{2^*_{\alpha}}(w)\int\limits_{\Omega\times\{0\}}|w|^{2^*_{\alpha}}dx.
\end{equation}
Then, from \eqref{eq72}, \eqref{eq84*} and The H\"{o}lder inequality
\begin{equation}\label{eq74*}
\aligned
1-\lambda|f|_{L^{\infty}}\,|w|^{-(2^{*}_{\alpha}-q)}_{L^{2^*_{\alpha}}}
&\leq t^{2^*_{\alpha}-2}(w)=
\frac{\kappa_{\alpha}\int\limits_{\mathcal{C}_{\Omega}}y^{1-\alpha}|\nabla w|^{2}dxdy}
{\int\limits_{\Omega\times\{0\}}|w|^{2^*_{\alpha}}dx}\\
&=1+\frac{\lambda\int\limits_{\Omega\times\{0\}}f(x)|w|^{q}dx}{\int\limits_{\Omega\times\{0\}}|w|^{2^*_{\alpha}}dx}\\
&\leq 1+\lambda|f|_{L^{\infty}}\,|w|^{-(2^{*}_{\alpha}-q)}_{L^{2^*_{\alpha}}}
\endaligned\end{equation}

On the other hand, by \eqref{eq04} and \eqref{eq84}, we get
\begin{equation*}
\aligned
\int\limits_{\Omega\times\{0\}}|w|^{2^*_{\alpha}}dx
&>\frac{2-q}{2^*_{\alpha}-q} \kappa_{\alpha}\int\limits_{\mathcal{C}_{\Omega}}y^{1-\alpha}|\nabla w|^{2}dxdy\\
&\geq\frac{2-q}{2^*_{\alpha}-q} \kappa_{\alpha}S(\alpha,N)
\Big(\int\limits_{\Omega\times\{0\}}|w|^{2^*_{\alpha}}dx\Big)^{\frac{2}{2^*_{\alpha}}},\\
\endaligned\end{equation*}
that is
\begin{equation}\label{eq74**}
|w|_{L^{2^*_{\alpha}}}
>\Big(\frac{(2-q)\kappa_{\alpha}S(\alpha,N)}{2^*_{\alpha}-q}\Big)^{\frac{1}{2^*_{\alpha}-2}}.
\end{equation}
Hence, from \eqref{eq74**} and \eqref{eq74*}, we obtain \eqref{eq72*}.
This completes the proof.
\end{proof}

\begin{remark}\label{re3}
From \eqref{eq72*}, it is easy to see that $t(w)\to 1$ as $\lambda\to 0$.
\end{remark}

{\bf Proof the Theorem \ref{th02}.} Suppose that $\{\lambda_{n}\}$ is a sequence of positive number
such that $\lambda_{n}\to 0$ as $n\to +\infty$.
Let $w_{n}^{(1)}=w_{1,n}\in \mathcal{N}^{+}$ and $w_{n}^{(2)}=w_{2,n}\in \mathcal{N}^{-}$ are position
solutions of equation \eqref{eq02} corresponding to $\lambda=\lambda_{n}$. We have two following results.
\begin{itemize}
\item[(i)] By Remark \ref{r2}, for every $w_{n}^{(1)}\in\mathcal{N}^{+}$, we can conclude that $\|w_{n}^{(1)}\|_{H^{\frac{\alpha}{2}}_{0,L}(\mathcal{C}_{\Omega})}\to 0$ as $n\to \infty$.
\item[(ii)] By Lemma \ref{le9} and Remark \ref{re3}, for every $w_{n}^{(2)}\in\mathcal{N}^{-}$, there is
a unique $t(w_{n}^{(2)})>0$ such that $t(w_{n}^{(2)})\,w_{n}^{(2)}\in \mathcal{O}_{1}$, and $t(w_{n}^{(2)})\to 1$ as $n\to \infty$.
\end{itemize}

For case (ii). For each $w_{n}^{(2)}\in \mathcal{N}^{-}$, let
$$f(t)=J_{\mu}(tw_{n}^{(2)})=t^{2}\frac{\kappa_{\alpha}}{2}\int\limits_{\mathcal{C}_{\Omega}}
 y^{1-\alpha}|\nabla w^{(2)}_{n}|^{2}dxdy-t^{2^*_{\alpha}}\frac{\mu}{2^*_{\alpha}}\int\limits_{\Omega\times\{0\}} |w^{(2)}_{n}|^{2^*_{\alpha}}dx.$$
Since  $f(t)\to -\infty$ as $s\to \infty$, $\sup\limits_{t\geq 0}f(t)$ is achieved at
some $\widetilde{t}>0$ with $f'(\widetilde{t})=0$, which is
$$f'(\widetilde{t})=\widetilde{t}\Big(\| w_{n}^{(2)}\|^{2}_{H^{\frac{\alpha}{2}}_{0,L}(\mathcal{C}_{\Omega})}
-\widetilde{t}^{2^*_{\alpha}-2}\mu\int\limits_{\Omega\times\{0\}} |w_{n}^{(2)}|^{2^*_{\alpha}}dx\Big)=0.$$

Let
$$\widetilde{t}=\Big(\frac{\| w_{n}^{(2)}\|^{2}_{H^{\frac{\alpha}{2}}_{0,L}(\mathcal{C}_{\Omega})}}
{\mu\int\limits_{\Omega\times\{0\}} |w_{n}^{(2)}|^{2^*_{\alpha}}dx}\Big)^{\frac{1}{2^*_{\alpha}-2}}.$$
Then
$\widetilde{t}\,w_{n}^{(2)}\in \mathcal{O}_{\mu}$
and
$$\sup\limits_{t\geq 0}J_{\mu}(tw_{n}^{(2)})=J_{\mu}(\widetilde{t}\,w_{n}^{(2)})=
\frac{\alpha}{2N}\Big(\frac{\| w_{n}^{(2)}\|^{2}_{H^{\frac{\alpha}{2}}_{0,L}(\mathcal{C}_{\Omega})}}
{\mu\int\limits_{\Omega\times\{0\}} |w_{n}^{(2)}|^{2^*_{\alpha}}dx}\Big)^{\frac{N-\alpha}{2}}.$$

On the other hand, by H\"{o}lder inequality and Young inequality, for $\mu\in(0\,,\,1)$,
we have
\begin{equation*}\label{eq100}
\aligned
\int\limits_{\Omega\times\{0\}}f(x)|\widetilde{t}&w_{n}^{(2)}|^{q}dx
\leq |f|_{L^{\infty}}\Big(\int\limits_{\Omega\times\{0\}}
|\widetilde{t}w_{n}^{(2)}|^{2^*_{\alpha}}dx\Big)^{\frac{q}{2^*_{\alpha}}}\\
&\leq |f|_{L^{\infty}}\Big(\kappa_{\alpha}\,S(\alpha,N)\Big)^{-\frac{q}{2}}
\widetilde{t}^{q}\,\|w_{n}^{(2)}\|^{q}_{H^{\frac{\alpha}{2}}_{0,L}(\mathcal{C}_{\Omega})}\\
&\leq \frac{2-q}{2}\Big(|f|_{L^{\infty}}(\kappa_{\alpha}\,S(\alpha,N)\,\mu)^{-\frac{q}{2}}\Big)^{\frac{2}{2-q}}
+\mu\frac{q}{2}\Big(\widetilde{t}^{q}\,\|w_{n}^{(2)}\|^{q}_{H^{\frac{\alpha}{2}}_{0,L}(\mathcal{C}_{\Omega})}
\Big)^{\frac{2}{q}}\\
&=\frac{2-q}{2}\mu^{\frac{-q}{2-q}}
\Big(|f|_{L^{\infty}}(\kappa_{\alpha}\,S(\alpha,N))^{-\frac{q}{2}}\Big)^{\frac{2}{2-q}}
+\frac{\mu q}{2}\|\widetilde{t}w_{n}^{(2)}\|^{2}_{H^{\frac{\alpha}{2}}_{0,L}(\mathcal{C}_{\Omega})}.\\
\endaligned\end{equation*}
Then we get
\begin{equation}\label{eq101}
\aligned
J(\widetilde{t}\,w_{n}^{(2)})
&=\frac{1}{2}\|\widetilde{t}\,w_{n}^{(2)}\|^{2}_{H^{\frac{\alpha}{2}}_{0,L}(\mathcal{C}_{\Omega})}-
\frac{\lambda}{q}\int\limits_{\Omega\times\{0\}}f(x)|\widetilde{t}\,w_{n}^{(2)}|^{q}dx-
\frac{1}{2^*_{\alpha}}\int\limits_{\Omega\times\{0\}}|\widetilde{t}\,w_{n}^{(2)}|^{2^*_{\alpha}}dx\\
&\geq \frac{1-\lambda\mu}{2}\|\widetilde{t}\,w_{n}^{(2)}\|^{2}_{H^{\frac{\alpha}{2}}_{0,L}(\mathcal{C}_{\Omega})}
-\frac{\lambda(2-q)}{2q}\mu^{\frac{-q}{2-q}}
\Big(|f|_{L^{\infty}}(\kappa_{\alpha}\,S(\alpha,N))^{-\frac{q}{2}}\Big)^{\frac{2}{2-q}}\\
&-\frac{1}{2^*_{\alpha}}\int\limits_{\Omega\times\{0\}}|\widetilde{t}\,w_{n}^{(2)}|^{2^*_{\alpha}}dx\\
&=(1-\lambda\mu)
\Big(\frac{1}{2}\|\widetilde{t}\,w_{n}^{(2)}\|^{2}_{H^{\frac{\alpha}{2}}_{0,L}(\mathcal{C}_{\Omega})}
-\frac{\frac{1}{1-\lambda\mu}}{2^*_{\alpha}}\int\limits_{\Omega\times\{0\}}|\widetilde{t}\,w_{n}^{(2)}|^{2^*_{\alpha}}dx\Big)\\
&-\frac{\lambda(2-q)}{2q}\mu^{\frac{-q}{2-q}}
\Big(|f|_{L^{\infty}}(\kappa_{\alpha}\,S(\alpha,N))^{-\frac{q}{2}}\Big)^{\frac{2}{2-q}}\\
&=(1-\lambda\mu)J_{\frac{1}{1-\lambda\mu}}(\widetilde{t}\,w_{n}^{(2)})
-\frac{\lambda(2-q)}{2q}\mu^{\frac{-q}{2-q}}
\Big(|f|_{L^{\infty}}(\kappa_{\alpha}\,S(\alpha,N))^{-\frac{q}{2}}\Big)^{\frac{2}{2-q}}\\
&=(1-\lambda\mu)^{\frac{N-\alpha+2}{2}}J_{1}(\widetilde{t}\,w_{n}^{(2)})-
\frac{\lambda(2-q)}{2q}\mu^{\frac{-q}{2-q}}
\Big(|f|_{L^{\infty}}(\kappa_{\alpha}\,S(\alpha,N))^{-\frac{q}{2}}\Big)^{\frac{2}{2-q}}.\\
\endaligned\end{equation}
Therefore, corresponding to $\lambda=\lambda_{n}$,
 from \eqref{eq101}, Remark \ref{re3} and the fact
$$J(w_{n}^{(2)})<m_{J}+\frac{\alpha}{2N}\Big(\kappa_{\alpha}S(\alpha,N)\Big)^{\frac{N}{\alpha}},$$ we obtain

\begin{equation*}\label{eq102}
\aligned
&J_{1}(\widetilde{t}\,w_{n}^{(2)})\\
&\leq \Big(\frac{1}{1-\lambda_{n}\mu}\Big)^{\frac{N-\alpha+2}{2}}
\Big[J(\widetilde{t}\,w_{n}^{(2)})+\frac{\lambda_{n}(2-q)}{2q}\mu^{\frac{-q}{2-q}}
\Big(|f|_{L^{\infty}}(\kappa_{\alpha}\,S(\alpha,N))^{-\frac{q}{2}}\Big)^{\frac{2}{2-q}}\Big]\\
&<\Big(\frac{1}{1-\lambda_{n}\mu}\Big)^{\frac{N-\alpha+2}{2}}
\Big[m_{J}+\frac{\alpha}{2N}(\kappa_{\alpha}S(\alpha,N))^{\frac{N}{\alpha}}\\
&+\frac{\lambda_{n}(2-q)}{2q}\mu^{\frac{-q}{2-q}}
\Big(|f|_{L^{\infty}}(\kappa_{\alpha}\,S(\alpha,N))^{-\frac{q}{2}}\Big)^{\frac{2}{2-q}}\Big].
\endaligned
\end{equation*}
Since $m_{J}\to 0$, $\widetilde{t}\to 1$ as $n\to \infty$, it is easy to see that
$$\limsup\limits_{n\to \infty}J_{1}(w^{(2)}_{n})\leq \frac{\alpha}{2N}(\kappa_{\alpha}S(\alpha,N))^{\frac{N}{\alpha}}.$$
This tell us
$$\lim\limits_{n\to \infty}J_{1}(w^{(2)}_{n})= \frac{\alpha}{2N}(\kappa_{\alpha}S(\alpha,N))^{\frac{N}{\alpha}}.$$
We can conclude that $\{w_{n}^{(2)}\}$ is a minimizing sequence for $J_{1}$ in $\mathcal{O}_{1}$.
Then
$$
\kappa_{\alpha}\int\limits_{\mathcal{C}_{\Omega}}
 y^{1-\alpha}|\nabla w^{(2)}_{n}|^{2}dxdy-\int\limits_{\Omega\times\{0\}} |w^{(2)}_{n}|^{2^*_{\alpha}}dx\to 0
$$
and
$$J_{1}(w^{(2)}_{n})\to \frac{\alpha}{2N}(\kappa_{\alpha}S(\alpha,N))^{\frac{N}{\alpha}}$$
as $n\to \infty$. This implies that $\{w^{(2)}_{n}\}$ is a $(PS)_{c}$-sequence for $J_{1}$
at level $c=\frac{\alpha}{2N}(\kappa_{\alpha}S(\alpha,N))^{\frac{N}{\alpha}}$.
Clearly, $\{w^{(2)}_{n}\}$ is bounded, and
then there exists a subsequence $\{w^{(2)}_{n}\}$
and $w_{0}\in H^{\frac{\alpha}{2}}_{0,L}(\mathcal{C}_{\Omega})$ such that
$$w^{(2)}_{n}\rightharpoonup w_{0} \quad \text{weakly in}\,\,
 H^{\frac{\alpha}{2}}_{0,L}(\mathcal{C}_{\Omega}). $$
Since $\Omega$ is bounded, we have $w_{0}=0$. Moreover, by the concentration-compactness
principle (see Theorem 6 of \cite{ga2013}),
there exist two sequence $\{x_{n}\}\subset \Omega$, $\{R_{n}\}\subset \mathbb{R}^{+}$
 such that $R_{n}\to \infty$ as $n\to\infty$ and
 $$\|\text{tr}_{\Omega}w_{n}^{(2)}-R_{n}^{\frac{N-\alpha}{2}}u(R_{n}(x-x_{n}))\|_{H^{\frac{\alpha}{2}}_{0}(\Omega)}\to 0\quad\text{as}\,\,n\to \infty.$$
 This completes the proof of Theorem \ref{th02}.

\end{document}